\documentclass{amsart}
\usepackage[utf8]{inputenc}
\usepackage{amsmath}
\usepackage{amssymb}
\usepackage{mathrsfs}
\usepackage{graphicx}

\usepackage[colorlinks]{hyperref}
\usepackage[svgnames, dvipsnames, usenames]{xcolor}
\hypersetup{urlcolor = RedViolet, linkcolor = RoyalBlue, citecolor = ForestGreen}

\usepackage{tikz}
\usepackage{partmac}

\def\ppen{\penalty 300 }
\let\col=\colon
\def\colon{\col\ppen}

\theoremstyle{plain} 
\newtheorem{thm}{Theorem}[section]
\newtheorem*{thm*}{Theorem}
\newtheorem{prop}[thm]{Proposition}
\newtheorem{lem}[thm]{Lemma}
\newtheorem{cor}[thm]{Corollary}

\newtheorem{obs}[thm]{Observation}
\theoremstyle{definition}
\newtheorem{defn}[thm]{Definition}
\newtheorem{rem}[thm]{Remark}
\newtheorem{ex}[thm]{Example}
\newtheorem{quest}[thm]{Question}

\newtheorem{alg}[thm]{Algorithm}
\numberwithin{equation}{section}

\renewcommand{\theta}{\vartheta}
\renewcommand{\phi}{\varphi}
\renewcommand{\epsilon}{\varepsilon}
\renewcommand{\subset}{\subseteq}
\renewcommand{\supset}{\supseteq}

\newcommand{\N}{\mathbb N}
\newcommand{\Z}{\mathbb Z}

\newcommand{\C}{\mathbb C}
\newcommand{\staralg}{\mathop{\rm\ast\mathchar `\-alg}}
\DeclareMathOperator{\Mor}{Mor}
\DeclareMathOperator{\id}{id}
\DeclareMathOperator{\sgn}{sgn}
\DeclareMathOperator{\Aut}{Aut}
\DeclareMathOperator{\Irr}{Irr}
\DeclareMathOperator{\spanlin}{span}
\DeclareMathOperator{\Rep}{Rep}
\newcommand{\Cat}{\mathscr{C}}
\newcommand{\RCat}{\mathfrak{C}}
\newcommand{\Lin}{\mathscr{L}}
\newcommand{\Part}{\mathscr{P}}
\newcommand{\Partlin}{\mathsf{Part}}
\newcommand{\Pair}{\mathsf{Pair}}
\newcommand{\A}{\mathcal{A}}
\newcommand{\Cay}{\mathrm{Cay}}
\newcommand{\F}{\mathcal{F}}
\newcommand{\adet}{\mathop{\widebreve{\rm det}}\nolimits}
\renewcommand{\det}{\mathop{\rm det}\nolimits}
\newcommand{\sqprod}{\mathrel{\square}}
\newcommand{\Olg}{\mathscr{O}}
\newcommand{\brewedge}{\mathop{\breve\wedge}}

\newcommand{\forkpart}{\Partition{
\Pblock 0to0.5:1,2
\Psingletons 1to0.5:1.5
}}

\newcommand{\mergepart}{\Partition{
\Pblock 1to0.5:1,2
\Psingletons 0to0.5:1.5
}}

\newcommand{\mergeone}{\Partition{
\Pblock 1 to 0.7:1,2
\Pline  (1.5,0.4) (1.5,0.7)
\Ptext (1,1.3) {$\scriptscriptstyle 1$}
\Ptext (2,1.3) {$\scriptscriptstyle 1$}
\Ptext (1.5,0.1) {$\scriptscriptstyle 1$}
}}

\newcommand{\connecterone}{\Partition{
\Pblock 0.4 to 0.6:1,2
\Pblock 1 to 0.8:1,2
\Pline  (1.5,0.6) (1.5,0.8)
\Ptext (1,1.3) {$\scriptscriptstyle 1$}
\Ptext (2,1.3) {$\scriptscriptstyle 1$}
\Ptext (1,0.1) {$\scriptscriptstyle 1$}
\Ptext (2,0.1) {$\scriptscriptstyle 1$}
}}

\newcommand{\fournum}{\Partition{
\Pblock 0.4to0.7:1,2,3,4
\Ptext (1,0.1) {$\scriptscriptstyle 2$}
\Ptext (2,0.1) {$\scriptscriptstyle 2$}
\Ptext (3,0.1) {$\scriptscriptstyle 2$}
\Ptext (4,0.1) {$\scriptscriptstyle 2$}
}}

\newcommand{\blockpart}{\Partition{
\Pblock 1to0.5:1,2,4,5
\Ptext (3,0.7) {$\,\cdots$}
}}

\newcommand{\immersepart}{\Partition{
\Pline (2,1) (1,0)
\Pline (3,1) (4,0)
\Pblock 0to0.3:2,3
}}

\def\Lnrule{\vrule height12pt depth 3pt}
\def\Lnmatrix#1{\left(
\vcenter{\offinterlineskip\halign{
\hfil\enskip$##$\enskip\hfil\Lnrule&
\hfil\enskip$##$\enskip\hfil&\hfil\enskip$##$\enskip\hfil&\hfil\enskip$##$\enskip\hfil\Lnrule&
\hfil\enskip$##$\enskip\hfil&\hfil\enskip$##$\enskip\hfil&\hfil\enskip$##$\enskip\hfil\Lnrule&
\hfil\enskip$##$\enskip\hfil&\rlap{\kern1em$\scriptstyle ##$}\cr
#1
\crcr
}}
\right)}

\def\Pa (#1,#2){
	\pgfpathellipse{\pgfpointxy{#1}{#2}}{\pgfpoint{0.2em}{0em}}{\pgfpoint{0em}{0.1em}}
	\pgfsetfillcolor{white}
	\pgfusepath{stroke,fill}
}

\coldef\Acol\Pa

\newcommand{\PAid}{\Partition{
\Pline (0.8,1) (0.8,0)
\Pline (1.2,1) (1.2,0)
}[a/a]}

\newcommand{\PApair}{\Partition{
\Pblock 0to0.3:1.2,1.8
\Pblock 0to0.6:0.8,2.2
}[aa]}

\newcommand{\PAthree}{\Partition{
\Pblock 0to0.3:1.2,1.8
\Pblock 0to0.3:2.2,2.8
\Pblock 0to0.6:0.8,3.2
}[aaa]}

\newcommand{\PAfour}{\Partition{
\Pblock 0to0.3:1.2,1.8
\Pblock 0to0.3:2.2,2.8
\Pblock 0to0.3:3.2,3.8
\Pblock 0to0.6:0.8,4.2
}[aaaa]}

\newcommand{\PAfive}{\Partition{
\Pblock 0to0.3:1.2,1.8
\Pblock 0to0.3:2.2,2.8
\Pblock 0to0.3:3.2,3.8
\Pblock 0to0.3:4.2,4.8
\Pblock 0to0.6:0.8,5.2
}[aaaaa]}

\newcommand{\PAimmerse}{\Partition{
\Pline (1.3,1) (0.8,0)
\Pline (1.7,1) (2.2,0)
\Pblock 0to0.3:1.2,1.8
\Ppoint1 \Pa:1.5
\Ppoint0 \Pa:1,2
}}

\newcommand{\PAconnecter}{\Partition{
\Pline (0.8,1) (0.8,0)
\Pline (2.2,1) (2.2,0)
\Pblock 0to0.3:1.2,1.8
\Pblock 1to0.7:1.2,1.8
}[aa/aa]}

\newcommand{\PAdoubleimmerse}{\Partition{
\Pline (1.3,1) (0.8,0)
\Pline (2.7,1) (3.2,0)
\Pblock 0to0.3:1.2,1.8
\Pblock 0to0.3:2.2,2.8
\Pblock 1to0.7:1.7,2.3
\Ppoint1 \Pa:1.5,2.5
\Ppoint0 \Pa:1,2,3
}}

\newcommand{\PAcross}{\Partition{
\Pline (0.8,1) (1.8,0)
\Pline (1.2,1) (2.2,0)
\Pline (1.8,1) (0.8,0)
\Pline (2.2,1) (1.2,0)
}[aa/aa]}

\newcommand{\PAinnercross}{\Partition{
\Pline (0.8,1) (0.8,0)
\Pline (1.2,1) (1.8,0)
\Pline (1.8,1) (1.2,0)
\Pline (2.2,1) (2.2,0)
}[aa/aa]}

\newcommand{\PAoutercross}{\Partition{
\Pblock 0to0.2:1.2,1.8
\Pblock 1to0.8:1.2,1.8
\Pline (0.8,1) (2.2,0)
\Pline (2.2,1) (0.8,0)
}[aa/aa]}

\newcommand{\PApairproj}{\Partition{
\Pblock 0to0.2:1.2,1.8
\Pblock 0to0.4:0.8,2.2
\Pblock 1to0.8:1.2,1.8
\Pblock 1to0.6:0.8,2.2
}[aa/aa]}

\newcommand{\PAAbb}{\Partition{
\pgfsetdash{{1pt}{1pt}}{0pt}
\Pblock 1to0.7:1,2
\pgfsetdash{{0.5pt}{0.8pt}}{0pt}
\Pblock 0to0.3:1,2
}}

\newcommand{\PAABB}{\Partition{
\pgfsetdash{{1pt}{1pt}}{0pt}
\Pblock 1to0.7:1,2
\Pblock 0to0.3:1,2
}}

\newcommand{\PaBaB}{\Partition{
\pgfsetdash{{1pt}{1pt}}{0pt}
\Pline (1,1) (2,0)
\pgfsetdash{{0.5pt}{0.8pt}}{0pt}
\Pline (2,1) (1,0)
}}

\newcommand{\PaBBa}{\Partition{
\pgfsetdash{{1.5pt}{1pt}}{0pt}
\Pline (1,1) (1,0)
\pgfsetdash{{0.4pt}{0.8pt}}{0pt}
\Pline (2,1) (2,0)
}}

\makeatletter
\def\widebreve{\mathpalette\wide@breve}
\def\wide@breve#1#2{\sbox\z@{$#1#2$}%
     \mathop{\vbox{\m@th\ialign{##\crcr
\kern0.08em\brevefill#1{0.8\wd\z@}\crcr\noalign{\nointerlineskip}%
                    $\hss#1#2\hss$\crcr}}}\limits}
\def\brevefill#1#2{$\m@th\sbox\tw@{$#1($}%
  \hss\resizebox{#2}{\wd\tw@}{\rotatebox[origin=c]{90}{\upshape(}}\hss$}
\makeatletter

\begin{document}
\title{Quantum symmetries of Cayley graphs of abelian groups}
\author{Daniel Gromada}
\address{Czech Technical University in Prague, Faculty of Electrical Engineering, Department of Mathematics, Technická 2, 166 27 Praha 6, Czechia}
\email{gromadan@fel.cvut.cz}
\thanks{I would like to thank to Simon Schmidt for discussions about the quantum symmetries of folded and halved hypercube graphs. I also thank to Christian Voigt for spotting a mistake in an earlier version of this manuscript.}
\thanks{This work was supported by the project OPVVV CAAS CZ.02.1.01/0.0/0.0/16\_019/0000778}
\date{\today}
\subjclass[2020]{20G42 (Primary); 05C25, 18M25 (Secondary)}
\keywords{Cayley graph, hypercube graph, Hamming graph, quantum symmetry}

\begin{abstract}
We study Cayley graphs of abelian groups from the perspective of quantum symmetries. We develop a general strategy for determining the quantum automorphism groups of such graphs. Applying this procedure, we find the quantum symmetries of the halved cube graph, the folded cube graph and the Hamming graphs.
\end{abstract}

\maketitle
\section*{Introduction}
The adjective \emph{quantum} in the title of this article refers to non-commutative geometry. In 1987, Woronowicz introduced the notion of \emph{compact quantum groups} \cite{Wor87} (following earlier work of Drinfeld and Jimbo) as the generalization of compact groups in the non-commutative geometry. This allowed to study symmetries of different objects not only in terms of the classical theory of symmetry groups, but also in terms of quantum groups. A remarkable property of graphs is that although they are classical objects, they often possess not only classical symmetries, but also the quantum ones. This was first noticed by Wang \cite{Wan98}, who defined \emph{free quantum symmetric group} $S_N^+$ as the quantum group of symmetries of a finite space of $N$ points. This quantum group is much larger than the classical group $S_N$ if $N\ge 4$.

Studying quantum symmetries of graphs started with the work of Bichon \cite{Bic03} and later Banica \cite{Ban05}, who gave the definition of the \emph{quantum automorphism group} of a graph. Since then, many authors worked on determining the quantum automorphism groups of different graphs. Worth mentioning is the joint publication of the mentioned two authors \cite{BB07} determining quantum symmetries of vertex transitive graphs up to 11 vertices and a recent extensive work of Schmidt, which is summarized in his PhD thesis \cite{SchThesis}.

An important tool for studying compact quantum groups was introduced again by Woronowicz in \cite{Wor88} -- the monoidal $*$-category of representations and intertwiners. Woronowicz formulated a generalization of the so-called Tannaka--Krein duality: He proved that a compact quantum group is uniquely determined by its representation category. A very useful result then came with the work of Banica and Speicher \cite{BS09}, who showed how to model those intertwiners using combinatorial objects -- partitions.

This formalism can also be used when working with quantum automorphisms of graphs. Given a graph $X$, its quantum automorphism group can be defined as the unique compact matrix quantum group $G$, whose representation category is generated by the intertwiners $T_{\pairpart}^{(N)}$, $T_{\mergepart}^{(N)}$ and $A_X$, where $A_X$ is the adjacency matrix of $X$.

In this paper, we study Cayley graphs of abelian groups. We use the intertwiner formalism to formulate a general algorithm for determining the quantum automorphism groups of such graphs. The result is presented in Section~\ref{sec.Cayley} as Algorithm~\ref{Algorithm}. It is based on the idea, which was already used in \cite{BBC07} to determine the quantum symmetries of the hypercube graph, namely that the Fourier transform on the underlying group $\Gamma$ diagonalizes the adjacency matrix of the Cayley graph of this group. The case of the hypercube graph is presented in Section~\ref{sec.hypercube} as a motivating example.

As a side remark, let us mention the work of Chassaniol \cite{Cha19}, who uses the intertwiner approach to determine quantum symmetries of some circulant graphs, i.e.~ Cayley graphs of the cyclic groups. But apart from using the intertwiners, his techniques are different from ours.

Subsequently, we use our algorithm to determine the quantum automorphism groups of certain Cayley graphs, which were not known before. This constitutes the main result of this article, which can be summarized as follows:
\begin{thm*}
We determine the quantum automorphism groups of the following graphs.
\begin{enumerate}
\renewcommand{\theenumi}{\alph{enumi}}
\item {}[Theorem~\ref{T.demicube}] For $n\neq 1,3$, the quantum symmetries of the halved hypercube graph $\frac{1}{2}Q_{n+1}$ are described by the anticommutative special orthogonal group $SO_{n+1}^{-1}$.
\item {}[Theorem~\ref{T.fcube}] For $n\neq 1,3$, the quantum symmetries of the folded hypercube graph $FQ_{n+1}$ are described by the projective anticommutative orthogonal group $PO_{n+1}^{-1}$.
\item {}[Theorem~\ref{T.Hamming}] For $m\neq 1,2$, the quantum symmetries of the Hamming graph $H(n,m)$ are described by the wreath product $S_m^+\wr S_n$.
\end{enumerate}
\end{thm*}

Note that quantum symmetries of the folded hypercube and the Hamming graphs were already studied before \cite{Sch20dt,Sch20fq}, but determining the quantum automorphism group for a general value of the parameters was left open.

\section{Preliminaries}
\label{sec.prelim}

In this section we recall the basic notions of compact matrix quantum groups and Tannaka--Krein duality. For a~more detailed introduction, we refer to the monographs \cite{Tim08,NT13}.

\subsection{Notation}
In this work, we will often work with operators between some tensor powers of some vector spaces. Therefore, we adopt the ``physics notation'' with upper and lower indices for entries of these ``tensors''. That is, given $T\colon V^{\otimes k}\to V^{\otimes l}$ for some $V=\C^N$, we denote
$$T(e_{i_1}\otimes\cdots\otimes e_{i_k})=\sum_{j_1,\dots,j_l=1}^NT^{j_1\cdots j_l}_{i_1\cdots i_k}(e_{j_1}\otimes\cdots\otimes e_{j_l})$$
We will sometimes shorten the notation and write $T_\mathbf{i}^\mathbf{j}$ using multiindices $\mathbf{i}=(i_1,\dots,i_k)$, $\mathbf{j}=(j_1,\dots,j_l)$.


\subsection{Compact matrix quantum groups}
\label{secc.qgdef}
A~\emph{compact matrix quantum group} is a pair $G=(A,u)$, where $A$ is a~$*$-algebra and $u=(u^i_j)\in M_N(A)$ is a matrix with values in $A$ such that
\begin{enumerate}
\item the elements $u^i_j$, $i,j=1,\dots, N$ generate $A$,
\item the matrices $u$ and $u^t$ ($u$ transposed) are similar to unitary matrices,
\item the map $\Delta\colon A\to A\otimes A$ defined as $\Delta(u^i_j):=\sum_{k=1}^N u^i_k\otimes u^k_j$ extends to a~$*$-homomorphism.
\end{enumerate}

Compact matrix quantum groups introduced by Woronowicz \cite{Wor87} are generalizations of compact matrix groups in the following sense. For a~matrix group $G\subseteq M_N(\C)$, we define $u^i_j\colon G\to\C$ to be the coordinate functions $u^i_j(g):=g^i_j$. Then we define the \emph{coordinate algebra} $A:=\Olg(G)$ to be the algebra generated by $u^i_j$. The pair $(A,u)$ then forms a compact matrix quantum group. The so-called \emph{comultiplication} $\Delta\colon \Olg(G)\to \Olg(G)\otimes \Olg(G)$ dualizes matrix multiplication on $G$: $\Delta(f)(g,h)=f(gh)$ for $f\in \Olg(G)$ and $g,h\in G$.

Therefore, for a~general compact matrix quantum group $G=(A,u)$, the algebra~$A$ should be seen as the algebra of non-commutative functions defined on some non-commutative compact underlying space. For this reason, we often denote $A=\Olg(G)$ even if $A$ is not commutative. Actually, $A$ also has the structure of a~Hopf $*$-algebra. In addition, we can also define the C*-algebra $C(G)$ as the universal C*-completion of $A$, which can be interpreted as the algebra of continuous functions of $G$. The matrix $u$ is called the \emph{fundamental representation} of~$G$.

A compact matrix quantum group $H=(\Olg(H),v)$ is a~\emph{quantum subgroup} of $G=(\Olg(G),u)$, denoted as $H\subseteq G$, if $u$ and $v$ have the same size and there is a~surjective $*$-homomorphism $\phi\colon \Olg(G)\to \Olg(H)$ sending $u^i_j\mapsto v^i_j$. We say that $G$ and $H$ are \emph{equal} if there exists such a $*$-isomorphism (i.e. if $G\subset H$ and $H\subset G$). We say that $G$ and $H$ are \emph{isomorphic} if there exists a $*$-isomorphism $\phi\colon \Olg(G)\to \Olg(H)$ such that $(\phi\otimes\phi)\circ\Delta_G=\Delta_H\circ\phi$. We will often use the notation $H\subseteq G$ also if $H$ is isomorphic to a quantum subgroup of $G$.

One of the most important examples is the quantum generalization of the orthogonal group -- the \emph{free orthogonal quantum group} $O_N^+$ defined by Wang in \cite{Wan95free} through the universal $*$-algebra
$$\Olg(O_N^+):=\staralg(u^i_j,\;i,j=1,\dots,N\mid u^i_j=u^{i*}_j,uu^t=u^tu=\id_{\C^N}).$$
Note that this example was then further generalized by Banica~\cite{Ban96} into the \emph{universal free orthogonal quantum group} $O^+(F)$, where $F\in M_N(\C)$ such that $F\bar F=\pm\id_{\C^N}$ and
$$\Olg(O^+(F)):=\staralg(u^i_j,\;i,j=1,\dots,N\mid\text{$u$ is unitary, }u=F\bar uF^{-1}),$$
where $[\bar u]^i_j=u^{i*}_j$.

\subsection{Representation categories and Tannaka--Krein reconstruction}
\label{secc.Rep}

For a~compact matrix quantum group $G=(\Olg(G),u)$, we say that $v\in M_n(\Olg(G))$ is a~representation of $G$ if $\Delta(v^i_j)=\sum_{k}v^i_k\otimes v^k_j$, where $\Delta$ is the comultiplication. The representation $v$ is called \emph{unitary} if it is unitary as a~matrix, i.e. $\sum_k v^i_kv^{j*}_k=\sum_k v^{k*}_iv^k_j=\delta_{ij}$. In particular, an element $a\in \Olg(G)$ is a one-dimensional representation if $\Delta(a)=a\otimes a$. Another example of a quantum group representation is the fundamental representation $u$.

We say that a representation $v$ of $G$ is \emph{non-degenerate} if $v$ is invertible as a matrix (in the classical group theory, we typically consider only non-degenerate representations). It is \emph{faithful} if $\Olg(G)$ is generated by the entries of~$v$. The meaning of this notion is the same as with classical groups: Given a non-degenerate faithful representation $v$, the pair $G'=(\Olg(G),v)$ is also a compact matrix quantum group, which is isomorphic to the original $G$.

A subspace $W\subset\C^n$ is called an invariant subspace of $v$ if the projection $P\colon\C^n\to\C^n$ onto $W$ commutes with $v$, that is, $Pv=vP$. This then defines the \emph{subrepresentation} $w:=vP=Pv$. However, $w$ as a representation is degenerate. If we need to express the subrepresentation as a non-degenerate representation, we had better consider a coisometry $U\colon\C^n\to\C^m$ with $P=U^*U$ and define $w':=UvU^*$. A representation $v$ is called irreducible if it has no non-trivial subrepresentations.

For two representations $v\in M_n(\Olg(G))$, $w\in M_m(\Olg(G))$ of $G$ we define the space of \emph{intertwiners}
$$\Mor(v,w)=\{T\colon \C^n\to\C^m\mid Tv=wT\}.$$
The set of all representations of a given quantum group together with those intertwiner spaces form a rigid monoidal $*$-category, which will be denoted by $\Rep G$.

Nevertheless, since we are working with compact \emph{matrix} quantum groups, it is more convenient to restrict our attention only to certain representations related to the fundamental representation. If we work with orthogonal quantum groups $G\subset O_n^+$ (or $G\subset O^+(F)$ in general), then it is enough to focus on the tensor powers $u^{\otimes k}$ since the entries of those representations already linearly span the whole $\Olg(G)$.

Considering $G=(\Olg(G),u)\subset O^+(F)$, $F\in M_N(\C)$, we define
$$\RCat_G(k,l):=\Mor(u^{\otimes k},u^{\otimes l})=\{T\colon (\C^N)^{\otimes k}\to(\C^N)^{\otimes l}\mid Tu^{\otimes k}=u^{\otimes l}T\}.$$

The collection of such linear spaces forms a rigid monoidal $*$-category with the monoid of objects being the natural numbers with zero $\N_0$.

\begin{rem}
The term \emph{rigidity} means that there exists a \emph{duality morphism} $R\in\RCat(0,2)$ such that $(R^*\otimes\id_{\C^N})(\id_{\C^N}\otimes R)=\id_{\C^N}$. For quantum groups $G\subset O^+(F)$, the duality morphism is given by $R^{ij}=F^j_i$.

An important feature of rigidity is the so-called \emph{Frobenius reciprocity}, which basically means that the whole category $\RCat$ is determined by the spaces $\RCat(0,k)$, $k\in\N_0$ since $\RCat$ is closed under certain \emph{rotations}.
\end{rem}

Conversely, we can reconstruct any compact matrix quantum group from its representation category \cite{Wor88,Mal18}.

\begin{thm}[Woronowicz--Tannaka--Krein]
\label{T.TK}
Let $\RCat$ be a rigid monoidal $*$-category with $\N_0$ being the set of objects and $\RCat(k,l)\subset\Lin((\C^N)^{\otimes k},(\C^N)^{\otimes l})$. Then there exists a unique orthogonal compact matrix quantum group $G$ such that $\RCat=\RCat_G$. We have $G\subset O^+(F)$ with $F^j_i=R^{ij}$, where $R$ is the duality morphism of $\RCat$.
\end{thm}

We can write down the associated quantum group very concretely. The relations satisfied in the algebra $\Olg(G)$ will be exactly the intertwining relations:
$$\Olg(G)=\staralg(u^i_j,\;i,j=1,\dots,N\mid u=F\bar uF^{-1},\; Tu^{\otimes k}=u^{\otimes l}T\;\forall T\in\RCat(k,l)).$$

We say that $S$ is a generating set for a representation category $\RCat$ if $\RCat$ is the smallest monoidal $*$-category satisfying the assumptions of Theorem~\ref{T.TK} that contains $S$. We use the notation $\RCat=\langle S\rangle_N$ (it is important to specify the dimension $N$ of the vector space $V=\C^N$ associated to the object 1). If we know such a generating set, it is enough to use the generators for our relations:
$$\Olg(G)=\staralg(u^i_j,\;i,j=1,\dots,N\mid u=F\bar uF^{-1},\; Tu^{\otimes k}=u^{\otimes l}T\;\forall T\in S(k,l)).$$

\subsection{Partitions}
\label{secc.part}
Representation categories of homogeneous orthogonal quantum groups, that is, those $G$ such that $S_N\subset G\subset O_N^+$ are conveniently described using partitions. A \emph{partition} $p\in\Part(k,l)$ is a decomposition of $k$ \emph{upper} and $l$ \emph{lower} points into non-empty disjoint subsets called \emph{blocks}. For instance,
$$
p=
\BigPartition{
\Pblock 0 to 0.25:2,3
\Pblock 1 to 0.75:1,2,3
\Psingletons 0 to 0.25:1,4
\Pline (2.5,0.25) (2.5,0.75)
}
\qquad
q=
\BigPartition{
\Psingletons 0 to 0.25:1,4
\Psingletons 1 to 0.75:1,4
\Pline (2,0) (3,1)
\Pline (3,0) (2,1)
\Pline (2.75,0.25) (4,0.25)
}
$$

Given any partition $p\in\Part(k,l)$, we define a linear map $T_p^{(N)}\colon (\C^N)^{\otimes k}\to(\C^N)^{\otimes l}$, whose entries are given as ``blockwise Kronecker delta'' -- we label the $k$ upper points by indices $i_1,\dots,i_k$ and the $l$ lower points by indices $j_1,\dots,j_l$ and define $[T_p^{(N)}]^{j_1,\dots,j_l}_{i_1,\dots,i_k}$ to be one if and only if for any given block of $p$ all the corresponding indices are equal. For instance, working with the example above, we may write
$$[T_p^{(N)}]^{j_1j_2j_3j_4}_{i_1i_2i_3}=\delta_{i_1i_2i_3j_2j_3},\quad
  [T_q^{(N)}]^{j_1j_2j_3j_4}_{i_1i_2i_3i_4}=\delta_{i_2j_3j_4}\delta_{i_3j_2}.$$

\begin{thm}[\cite{Jon94}]
\label{T.Jones}
It holds that $\RCat_{S_N}(k,l)=\spanlin\{T_p^{(N)}\mid p\in\Part(k,l)\}$ for every $N\in\N$.
\end{thm}

We define $\Partlin_N(k,l):=\spanlin\Part(k,l)$ the space of formal linear combinations of partitions. On this collection of linear spaces, we may define the structure of a monoidal $*$-category in terms of simple pictorial manipulations (see e.g.\ \cite{GWintsp} for details) such that they respect the category structure in $\RCat_{S_n}$. In other words, the mapping $T^{(N)}\colon\Partlin_n\to\RCat_{S_n}$ is a monoidal unitary functor. (Note that passing to the linear spaces is often omitted, but in this article, we need them.)

As a consequence, any homogeneous quantum group $O_N^+\supset G\supset S_N$ can be described using some diagrammatic category of partitions. Thanks to the Tannaka--Krein duality, we also have the converse -- any category of partitions defines a some homogeneous compact matrix quantum group \cite{BS09}. For more information, see the survey \cite{Web17} or the author's PhD thesis \cite{GroThesis}.

Finally, let us mention that when working with anticommutative deformations of groups (see the next section), then it might (although sometimes might not) be convenient to use a deformed functor. Let $p\in\Part(k,l)$ be a partition that does not contain any block of odd size. Then we define a linear operator $\breve T_p^{(N)}$ by
$$[\breve T_p^{(N)}]^\mathbf{j}_\mathbf{i}=\sigma_\mathbf{i}\sigma_\mathbf{j}[T_p^{(N)}]^\mathbf{j}_\mathbf{i},$$
where $\sigma_\mathbf{i}$ is a certain sign function: Given a multiindex $\mathbf{i}=(i_1,\dots,i_k)$, we count the number of pairs $(k,l)$ such that $k<l$, but $i_k>i_l$. If this number is odd, then $\sigma_\mathbf{i}=-1$; otherwise $\sigma_\mathbf{i}=1$. See \cite[Section~7]{GWgen}.

\subsection{Anticommutative deformations}
In this work, we will often work with certain anticommutative deformations of classical groups. We say that a~matrix $u$ has \emph{anticommutative entries} if the following relations hold
$$u^i_ku^j_k=-u^j_ku^i_k,\quad u^k_iu^k_j=-u^k_ju^k_i,\quad u^i_ku^j_l=u^j_lu^i_k$$
assuming $i\neq j$ and $k\neq l$.

As an example, let us mention the \emph{anticommutative orthogonal quantum group}
$$\Olg(O_N^{-1})=\staralg(u^i_j\mid\text{$u=\bar u$, $u$ orthogonal, $u$ anticommutative}).$$

There is a whole theory about $q$-deformations of classical groups, where taking $q=1$ gives the classical case and $q=-1$ gives usually the anticommutative one, see \cite{KS97} for more details.

\subsection{Exterior products -- classical case}

In this section, we would like to recall the definition of exterior products in connection with the representation theory of classical groups.

Let $V$ be a vector space. We define the \emph{exterior product} $V\wedge V$ and, more generally, the exterior powers $\Lambda_k(V)=V^{\wedge k}$ as follows. $V^{\wedge k}$ is the vector subspace of $V^{\otimes k}$ generated by the elements
$$v_1\wedge\cdots\wedge v_k={1\over k!}\sum_{\sigma\in S_k}\sgn(\sigma)v_{\sigma(1)}\otimes\cdots\otimes v_{\sigma(k)}$$
We denote by $\A_k\colon V^{\otimes k}\to V^{\wedge k}$ the coisometry mapping $v_1\otimes\cdots\otimes v_k\mapsto v_1\wedge\cdots\wedge v_k$ and call it the \emph{antisymmetrizer}. (That is, $\A_k^*\A_k\colon V^{\otimes k}\to V^{\otimes k}$ is the projection onto $V^{\wedge k}$ taken as a subspace of $V^{\otimes k}$.)

The motivation for such a definition is the following.

\begin{prop}
Let $G$ be any group and $V$ some $G$-module. Then $V^{\wedge k}$ is always a submodule of $V^{\otimes k}$.
\end{prop}
In particular, we can consider $G=O_n$ acting on $V=\C^n$ by standard matrix multiplication. Let us prove this statement from a quantum group point of view.
\begin{proof}
Let $u\in C(G)\otimes GL(V)$ be the representation of $G$ on $V$. We need to prove that $u^{\otimes k}\A_k^*\A_k=\A_k^*\A_k u^{\otimes k}$. Let us express both sides in coordinates.
\begin{align*}
[u^{\otimes k}\A_k^*\A_k]^\mathbf{i}_\mathbf{j}&={1\over k!}\sum_{\sigma\in S_k}\sgn(\sigma)u^{i_1}_{\sigma^{-1}(j_1)}\cdots u^{i_k}_{\sigma^{-1}(j_k)}\\
[\A_k^*\A_k u^{\otimes k}]^\mathbf{i}_\mathbf{j}&={1\over k!}\sum_{\sigma\in S_k}\sgn(\sigma)u^{\sigma(i_1)}_{j_1}\cdots u^{\sigma(i_k)}_{j_k} 
\end{align*}
The terms $u^{i_1}_{\sigma^{-1}(j_1)}\cdots u^{i_k}_{\sigma^{-1}(j_k)}$ and $u^{\sigma(i_1)}_{j_1}\cdots u^{\sigma(i_k)}_{j_k}$ differ just by reordering of the factors. Since we assume that $G$ is a classical group, the entries of $u$ are commutative, so the terms must be equal.
\end{proof}

We denote by
$$u^{\wedge k}:=\A_k u^{\otimes k}\A_k^*$$
the corresponding subrepresentation.


The dimension of the $k$-th exterior power $V^{\wedge k}$ equals $\binom{n}{k}$, where $n=\dim V$. The highest nonzero power is therefore the $n$-th, which is one-dimensional. Given a representation $u$ of some group $G$, the $n$-th exterior power of $u$ equals to the \emph{determinant} $u^{\wedge n}=\det u$.

\subsection{Exterior products -- anticommutative case}

The concept of exterior product does not work in general for quantum groups. Let us revise it here for the case of anticommutative deformations.

For this purpose we need to introduce some sort of ``anticommutative antisymmetrization''. This should be basically the same thing as the usual \emph{symmetrization}, but, in addition, we have to ``throw out the diagonal'' again. (Recall that classically we have $v_1\wedge\cdots\wedge v_k=0$ whenever $v_a=v_b$ for some $a,b$.)

We define $V^{\brewedge k}$ to be the vector subspace of $V^{\otimes k}$ generated by the elements
$$v_1\brewedge\cdots\brewedge v_k=
\begin{cases}
0&\text{if $v_a=v_b$ for some $a\neq b$}\\
\frac{1}{k!}\sum_{\sigma\in S_k}v_{\sigma(1)}\otimes\cdots\otimes v_{\sigma(k)}&\text{otherwise}
\end{cases}$$
We denote by $\breve\A\colon V^{\otimes k}\to V^{\brewedge k}$ the coisometry mapping $v_1\otimes\cdots\otimes v_k\mapsto v_1\brewedge\cdots\brewedge v_k$.

\begin{rem}
If we view anticommutative deformations as 2-cocycle twists of usual groups, then this procedure amounts to twisting the intertwiner $\A_k^*\A_k$ of $u$. In this sense, the rest of this subsection might be considered as obvious, but it does not harm to recall the facts explicitly.
\end{rem}

\begin{prop}
Consider $G\subset O_n^{-1}$ and denote by $u$ its fundamental representation. Then $u^{\otimes k}\breve\A_k^*\breve\A_k=\breve\A_k^*\breve\A_k u^{\otimes k}$.
\end{prop}
In other words, this means that $(\C^n)^{\brewedge k}$ is an invariant space of the representation $u^{\otimes k}$. We denote the corresponding subrepresentation by
$$u^{\brewedge k}:=\breve\A_k u^{\otimes k}\breve\A_k^*.$$
\begin{proof}
Let us write both sides of the equation entrywise.
\begin{align}
[u^{\otimes k}\breve\A_k^*\breve\A_k]^\mathbf{i}_\mathbf{j}&=
\begin{cases}
0&\text{if $j_a=j_b$ for some $a\neq b$}\\
\frac{1}{k!}\sum_{\sigma\in S_k}u^{i_1}_{\sigma^{-1}(j_1)}\cdots u^{i_k}_{\sigma^{-1}(j_k)}&\text{otherwise}
\end{cases}\\
[\breve\A_k^*\breve\A_k u^{\otimes k}]^\mathbf{i}_\mathbf{j}&=
\begin{cases}
0&\text{if $i_a=i_b$ for some $a\neq b$}\\
\frac{1}{k!}\sum_{\sigma\in S_k}u^{\sigma(i_1)}_{j_1}\cdots u^{\sigma(i_k)}_{j_k}&\text{otherwise}
\end{cases}
\end{align}
Notice again that $u^{i_1}_{\sigma^{-1}(j_1)}\cdots u^{i_k}_{\sigma^{-1}(j_k)}$ and $u^{\sigma(i_1)}_{j_1}\cdots u^{\sigma(i_k)}_{j_k}$ coincide up to ordering of the factors. If both $\bf i$ and $\bf j$ consist of mutually distinct indices, then the factors commute and hence the terms are equal. If $j_a=j_b$ for some $a\neq b$, then the factors $u^{\sigma(i_a)}_{j_a}$ and $u^{\sigma(i_b)}_{j_b}$ mutually anticommute. Consequently, the symmetrization $\sum_{\sigma}u^{\sigma(i_1)}_{j_1}\cdots u^{\sigma(i_k)}_{j_k}$ equals to zero. The same applies in the case when $i_a=i_b$ for some $a\neq b$.
\end{proof}

The dimension of the anticommutative exterior powers are again given by the binomial coefficients. So, taking the $n$-th power, we can define the \emph{anticommutative determinant} as follows
\begin{align*}
\adet u:=u^{\brewedge n}&=\sum_{\sigma\in S_n}u^{1}_{\sigma(1)}\cdots u^{n}_{\sigma(n)}=\sum_{\sigma\in S_n}u^{\pi(1)}_{\sigma(1)}\cdots u^{\pi(n)}_{\sigma(n)}\\
&=\sum_{\pi\in S_n}u^{\pi(1)}_{1}\cdots u^{\pi(n)}_{n}=\sum_{\pi\in S_n}u^{\pi(1)}_{\sigma(1)}\cdots u^{\pi(n)}_{\sigma(n)}.
\end{align*}
Since we assume the anticommutativity, all the factors of the terms always mutually commute. From this, the different ways to write down the determinant follow. The definition is very similar to the one of classical determinant --  we are only missing the sign of the permutation in the sum. Such an object is sometimes considered also in the classical theory of matrices, where it is called the \emph{permanent}. Since $\adet u$ is a one-dimensional representation of $O_n^{-1}$, it defines a quantum subgroup $SO_n^{-1}$ called the \emph{anticommutative special orthogonal quantum group}
$$\Olg(SO_n^{-1})=\staralg(u^i_j\mid\text{$u=\bar u$, $u$ orthogonal, $u$ anticommutative, $\adet u=1$}).$$

Note that the relation $\adet u=1$ can be seen also as an intertwiner relation $\breve\A_nu^{\otimes n}=\breve\A_n$. In other words $SO_n^{-1}$ is a quantum subgroup of $O_n^{-1}$ that was created by adding the intertwiner $\breve\A_n\in\Mor(u^{\otimes n},1)$ to its representation category. Similarly, $SO_n$ can be created from $O_n$ by imposing the intertwiner $\A_n\in\Mor(u^{\otimes n},1)$.

\subsection{Projective versions}
\label{secc.proj}

Let $G\subset O_N^+$ be an orthogonal compact matrix quantum group and denote by $u$ its fundamental representation. Then $u\otimes u$ is surely its representation, but its matrix entries may not generate the whole algebra $\Olg(G)$. Denote by $\Olg(PG)$ the $*$-subalgebra of $\Olg(G)$ generated by entries of $u\otimes u$, i.e.\ elements of the form $u^i_ju^k_l$. Then $PG:=(\Olg(PG),u\otimes u)$ is a compact matrix quantum group called the \emph{projective version} of $G$.

\begin{prop}
\label{P.POiso}
Consider $N\in\N$ odd. Then $PO_N^q\simeq SO_N^q$.
\end{prop}
For our work, we need only $q=+1$ (classical case) and $q=-1$ (anticommutative case). Nevertheless, the statement and its proof actually does not depend on $q$ and we could take any deformation of the orthogonal group here.
\begin{proof}
Denote by $u$ the fundamental representation of $O_N^q$, so that $u\otimes u$ is the fundamental representation of $PO_N^q$. Denote by $v$ the fundamental representation of $SO_N^+$.

We claim that there is a $*$-homomorphism $\alpha\colon \Olg(SO_N^q)\to \Olg(PO_N^q)$ mapping
$$v^i_j\mapsto u^i_j\det_q u$$
First, note that $u^i_j\det_q u$ is a polynomial of even degree in the entries of $u$ (since $N$ is odd, so the determinant is of odd degree) and hence $u^i_j\det_q u$ is indeed an element of $\Olg(PO_N^q)$. Secondly, it is easy to check that all relations of $SO_N^q$ are satisfied by the image. In particular the determinant equals to one since $\det_q(u^i_j\det_q u)_{i,j}=(\det_q u)^2=1$.

On the other hand, there is surely a $*$-homomorphism $\beta\colon \Olg(PO_N^q)\to \Olg(SO_N^q)$ mapping $u^i_ju^k_l\mapsto v^i_jv^k_l$ since this is nothing but the restriction of the quotient map $\Olg(O_N^q)\to \Olg(SO_N^q)$. Finally, it is easy to check that both $\beta\circ\alpha$ and $\alpha\circ\beta$ equal to the identity, so the maps must actually be isomorphisms.
\end{proof}


\section{Warm-up: Quantum symmetries of the classical hypercube}
\label{sec.hypercube}

In this section, we would like to revisit the result \cite[Theorem~4.2]{BBC07} saying that the quantum automorphism group of the $n$-dimensional hypercube graph is the anticommutative orthogonal group $O_n^{-1}$. Our aim is to explain the proof in slightly more detail and provide more explicit computations to make everything clear.

Throughout the whole paper, we are going to rely heavily on the theory of representation categories and we will express everything in terms of intertwiners. This may seem a bit clumsy in this particular case (in comparison with the approach of \cite{BBC07} for instance), but it will become very handy in the following sections, where we are going to study quantum symmetries of some other graphs. Using intertwiners, we will be able to formulate our approach in a very general way for arbitrary Cayley graphs of abelian groups.

\subsection{Quantum automorphism group of a graph}

We define the \emph{free symmetric quantum group} \cite{Wan98} $S_N^+=(C(S_N^+),u)$, where
$$\Olg(S_N^+)=\staralg(u^i_j,\;i,j=1,\dots,N\mid (u^i_j)^2=u^i_j=u^{i*}_j,\;\sum_k u^i_k=1).$$

It holds that $S_N^+$ describes all \emph{quantum symmetries} of the space of $N$ discrete points. What we mean by this is that $S_N^+$ is the largest quantum group that faithfully \emph{acts} on the space of $N$ points. Let us look on this property in even more detail.

Denote by $X_N=\{1,\dots,N\}$ the set of $N$ points. We can associate to $X_N$ the algebra of all functions $C(X_N)$, which has a basis $\delta_1,\dots,\delta_N$ of the canonical projections, that is, functions $\delta_i(j)=\delta_{ij}$. An \emph{action} of a quantum group $G$ on $X_N$ is described by a \emph{coaction} of the associated Hopf $*$-algebra, that is, a $*$-homomorphism $C(X_N)\to C(X_N)\otimes \Olg(G)$ satisfying some axioms.

Now, note that since $(\delta_i)$ is a linear basis of $C(X_N)$, the action of any compact quantum group on $X_N$ must be of the form $\delta_j\mapsto\sum_{i=1}^N \delta_i\otimes v^i_j$. The axioms of a coaction are now equivalent to the fact that $v$ is a representation of the acting quantum group $G$. The algebra $C(X_N)$ can be defined as the universal C*-algebra generated by $\delta_i$ satisfying the relations $\delta_i^2=\delta_i=\delta_i^*$ and $\sum_i\delta_i=1$. Now, it is easy to check that the requirement of the coaction being a $*$-homomorphism exactly corresponds to the defining relations of $C(S_N^+)$.

Alternatively, one can see the homomorphism condition also as some kind of an intertwiner relation. $S_N^+$ can be seen as a quantum subgroup of $O_N^+$ with respect to the relation $uT^{(N)}_{\mergepart}=T^{(N)}_{\mergepart}(u\otimes u)$, that is, requiring $T^{(N)}_{\mergepart}\in\Mor(u\otimes u,u)$. Here $T^{(N)}_{\mergepart}$ is a tensor $C^N\to\C^N\otimes\C^N$ defined by $[T^{(N)}_{\mergepart}]_{ij}^k=\delta_{ijk}$. See also \cite{Ban99,Ban02}.

Actually, it is easy to check that the partition $\pairpart$ defining $O_N^+$ together with $\mergepart$ defining $S_N^+\subset O_N^+$ generate all \emph{non-crossing partitions}, that is, partitions where the strings do not cross. Let us denote by $NC(k,l)$ the set of all partitions $p\in\Part(k,l)$, which are non-crossing. This provides a ``free quantum analogue'' to Theorem~\ref{T.Jones} of Jones.
\begin{prop}[\cite{BS09}]
It holds that $\RCat_{S_N^+}(k,l)=\spanlin\{T_p^{(N)}\mid p\in NC(k,l)\}$ for every $N\in\N$.
\end{prop}

Now let $X$ be a finite graph, so it has a finite set of vertices $V:=V(X)$. Let us number those vertices, so we can write $V=\{1,\dots,N\}$. The adjacency matrix of $X$ is then an $N\times N$ matrix $A_X$ with entries consisting of \emph{zeros} and \emph{ones} such that $[A_X]^j_i=1$ if and only if $(i,j)$ is an edge in $X$. If $X$ is undirected, then $A_X$ should be symmetric, otherwise it need not. If we have $[A_X]^i_i=1$, we say that $X$ has a loop at the vertex $i$. All concrete examples of graphs mentioned in this paper will be \emph{simple}, i.e.\ undirected and without loops, but the general considerations will hold also for the directed case with loops.

We say that a quantum group $G$ \emph{acts} on the graph $X$ through the coaction $\alpha\colon \delta_i\mapsto\sum_j\delta_j\otimes u^j_i$ if the coaction $\alpha$ commutes with the adjacency matrix, that is, $\alpha\circ A=(A\otimes\id)\circ\alpha$. Equivalently, this means the $uA=Au$. The quantum automorphism group of $X$ is defined to be the universal quantum group acting on $X$.

\begin{defn}[{\cite{Ban05}}]
Let $X$ be a graph on $N$ vertices. We define the \emph{quantum automorphism group} of $X$ to be the compact matrix quantum group $\Aut^+X=(\Olg(\Aut^+X),u)$ with
$$\Olg(\Aut^+X)=\staralg(u^i_j,\;i,j=1,\dots,N\mid (u^i_j)^2=u^i_j=u^{i*}_j,\;\sum_k u^i_k=1,\;Au=uA).$$
\end{defn}
Equivalently, it is determined by its representation category
$$\RCat_{\Aut^+X}=\langle T^{(N)}_{\pairpart},T^{(N)}_{\mergepart},A_X\rangle_N=\langle T_{\singleton}^{(N)},T_{\mergepart}^{(N)},A_X\rangle_N,$$
where $\singleton\in\Part(0,1)$ is the singleton partition. The equivalence of the two generating sets can be easily seen by noticing that $T^{(N)}_{\pairpart}=T^{(N)\,*}_{\mergepart}T^{(N)}_{\singleton}$ on one hand and $T^{(N)}_{\singleton}=T^{(N)}_{\mergepart}T^{(N)}_{\pairpart}$ on the other hand. See also \cite{Cha19}.

\subsection{The $n$-dimensional hypercube}
Consider a natural number $n\in\N$. The $n$-dimensional hypercube graph $Q_n$ is determined by the vertices and edges of an $n$-dimensional hypercube. It can be parametrized as follows.

The set of $N:=2^n$ vertices can be identified with the elements of the group $\Z_2^n$. We are going to denote these group elements by Greek letters $\alpha=(\alpha_1,\dots,\alpha_n)\in\Z_2^n$ with $\alpha_i\in\{0,1\}\simeq\Z_2$. We will denote the group operation as addition. We also denote the canonical generators by $\epsilon_i=(0,\dots,0,1,0,\dots,0)$.

Two vertices $\alpha,\beta\in V(Q_n)$ are connected by an edge if and only if they differ in exactly one of the $n$ indices, that is, if $\beta=\alpha+\epsilon_i$ for some $i\in\{1,\dots,n\}$. Equivalently, we can say that $Q_n$ is the Cayley graph of $\Z_2^n$ with respect to the generating set $\epsilon_1,\dots,\epsilon_n$. We can also express the edges via the adjacency matrix
$$[A_{Q_n}]^\alpha_\beta=
\begin{cases}
	1&\text{if $\beta=\alpha+\epsilon_i$ for some $i$,}\\
	0&\text{otherwise.}
\end{cases}$$

\begin{ex}[$n=3$]
We mention the example of the ordinary three-dimensional cube and its parametrization using triples of zero/one indices. Note that shifting the vertex by $\epsilon_i$, that is, flipping the $i$-th index, we move in the direction of the $i$-th dimension.
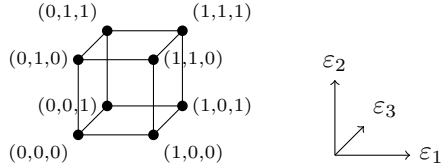
\def\vertex(#1){\fill (#1) circle (2pt)}
$$
\begin{tikzpicture}
\vertex (0,0,0) node [below left] {$\scriptstyle(0,0,0)$};
\vertex (1,0,0) node [below right] {$\scriptstyle(1,0,0)$};
\vertex (0,1,0) node [left] {$\scriptstyle(0,1,0)$};
\vertex (1,1,0) node [right] {$\scriptstyle(1,1,0)$};
\vertex (0,0,-1) node [left] {$\scriptstyle(0,0,1)$};
\vertex (1,0,-1) node [right] {$\scriptstyle(1,0,1)$};
\vertex (0,1,-1) node [above left] {$\scriptstyle(0,1,1)$};
\vertex (1,1,-1) node [above right] {$\scriptstyle(1,1,1)$};
\draw (0,0,0) -- (1,0,0) -- (1,1,0) -- (0,1,0) -- cycle;
\draw (0,0,-1) -- (1,0,-1) -- (1,1,-1) -- (0,1,-1) -- cycle;
\draw (0,0,0) -- (0,0,-1);
\draw (0,1,0) -- (0,1,-1);
\draw (1,0,0) -- (1,0,-1);
\draw (1,1,0) -- (1,1,-1);
\endtikzpicture
\qquad
\tikzpicture[->]
\draw (0,0) -- (1,0) node [right] {$\epsilon_1$};
\draw (0,0) -- (0,1) node [above] {$\epsilon_2$};
\draw (0,0,0) -- (0,0,-1) node [above right] {$\epsilon_3$};
\end{tikzpicture}
$$
Below, we show the corresponding adjacency matrix. Since it is not clear, how the elements of $\Z_2^n$ should be ordered, we also labelled the rows with the corresponding tuples (the columns are ordered the same way). We also indicated the division of the matrix into blocks with respect to the number of \emph{ones} in the tuple (essentially the distance from the $(0,0,0)$-vertex).
$$A_{Q_3}=\Lnmatrix{
0&1&1&1&0&0&0&0&(0,0,0)\cr\noalign{\hrule}
1&0&0&0&0&1&1&0&(1,0,0)\cr
1&0&0&0&1&0&1&0&(0,1,0)\cr
1&0&0&0&1&1&0&0&(0,0,1)\cr\noalign{\hrule}
0&0&1&1&0&0&0&1&(0,1,1)\cr
0&1&0&1&0&0&0&1&(1,0,1)\cr
0&1&1&0&0&0&0&1&(1,1,0)\cr\noalign{\hrule}
0&0&0&0&1&1&1&0&(1,1,1)\cr
}$$
\end{ex}

\subsection{Functions on the hypercube and Fourier transform}
We denote by $C(Q_n)$ the algebra of functions on the vertex set $V(Q_n)=\Z_2^n$. It has the canonical basis of $\delta$-functions $\delta_\alpha$ defined by $\delta_\alpha(\beta)=\delta_{\alpha\beta}$. 

We can also define the structure of a Hilbert space $l^2(Q_n)$ on the vector space of functions simply by putting $\langle f,g\rangle=\sum_\alpha \overline{f(\alpha)}g(\alpha)$. The basis $(\delta_\alpha)$ is then orthonormal with respect to this inner product.

There is another important basis on $l^2(Q_n)$ given by functions of the form
$$\tau_\mu=\tau_1^{\mu_1}\cdots \tau_n^{\mu_n},\qquad\hbox{where}\quad\mu=(\mu_1,\dots,\mu_n)\in\Z_2^n$$
and
$$\tau_i(\alpha)=(-1)^{\alpha_i},$$
so
$$\tau_\mu(\alpha)=(-1)^{\alpha_1\mu_1+\cdots+\alpha_n\mu_n}=(-1)^{\alpha\cdot\mu}.$$

First, note that the elements $\tau_\mu$ form a presentation of the group $\Z_2^n$ itself or, alternatively, of the group algebra $\C\Z_2^n$. That is, we have $\tau_\mu\tau_\nu=\tau_{\mu+\nu}$; indeed,
$$(\tau_\mu\tau_\nu)(\alpha)=(-1)^{\alpha\cdot\mu}(-1)^{\alpha\cdot\nu}=(-1)^{\alpha\cdot(\mu+\nu)}=\tau_{\mu+\nu}.$$

Secondly, note that the basis $(\tau_\mu)$ is orthogonal:
$$\langle\tau_\mu,\tau_\nu\rangle=\sum_\alpha (-1)^{\alpha\cdot(\mu+\nu)}=
\begin{cases}
	2^n&\text{if $\mu+\nu=0$, i.e. $\mu=\nu$,}\\
	0&\text{otherwise.}
\end{cases}$$

Let us explain more in detail the last equality, which will actually be useful also in subsequent computations.
\begin{lem}
For $\beta\in\Z_2^n$, it holds that
$$\sum_{\alpha\in\Z_2^n} (-1)^{\alpha\cdot\beta}=
\begin{cases}
	2^n&\text{if $\beta=0$,}\\
	0&\text{otherwise.}
\end{cases}$$
\end{lem}
\begin{proof}
If $\beta=0$, then $(-1)^{\alpha\cdot\beta}=(-1)^0=1$, so in the sum we are summing over $2^n$ \emph{ones}. Hence the result. If $\beta$ has some non-zero entry, say $\beta_i=1$, then for every $\alpha$, we can flip the $i$-th entry of $\alpha$ in order to flip the sign of $(-1)^{\alpha\cdot\beta}$. Thus, there is an equal amount of $+1$ as $-1$ in the sum, so they cancel out.
\end{proof}

We define the transformation matrix $\F\colon l^2(Q_n)\to l^2(Q_n)$ by $\F^\alpha_\mu=(-1)^{\alpha\cdot\mu}$ called the {\em Fourier transform}. It provides a transformation between the two bases: $\tau_\mu=\sum_\alpha\delta_\alpha \F^\alpha_\mu$. Thanks to the orthogonality property above, we have that $\F$ is, up to scaling, orthogonal. More precisely $\F^{-1}={1\over 2^n}\F^*$.

\begin{ex}[$n=3$]
Let us again look on the case $n=3$. The matrix $\F$ consists only of $\pm 1$ elements. For easier reading, we write only $+$ or $-$ in the matrix instead of $+1$ and $-1$. The order for the bases $(\delta_\alpha)$ and $(\tau_\mu)$ or, better to say, the order of the tuples $\alpha\in\Z_2^3$ are again indicated behind the matrix.

$$\F=\Lnmatrix{
+&+&+&+&+&+&+&+&(0,0,0)\cr\noalign{\hrule}
+&-&+&+&+&-&-&-&(1,0,0)\cr
+&+&-&+&-&+&-&-&(0,1,0)\cr
+&+&+&-&-&-&+&-&(0,0,1)\cr\noalign{\hrule}
+&+&-&-&+&-&-&+&(0,1,1)\cr
+&-&+&-&-&+&-&+&(1,0,1)\cr
+&-&-&+&-&-&+&+&(1,1,0)\cr\noalign{\hrule}
+&-&-&-&+&+&+&-&(1,1,1)\cr
}$$
\end{ex}

\subsection{Applying Fourier transform to the intertwiners}
\label{secc.fadj}

Recall that the quantum automorphism group of the hypercube $\Aut^+ Q_n$ should be the quantum subgroup of $S_N^+$ with respect to the intertwiner relation $uA=Au$, where we denote for short $A:=A_{Q_n}$. Equivalently, it is the quantum subgroup of $O_N^+$ with respect to relations $uT^{(N)}_{\mergepart}=T^{(N)}_{\mergepart}(u\otimes u)$ and $uA=Au$.

On the first sight, it is not clear, which quantum group these relations define. In order to see this, we first apply the Fourier transform to the intertwiners. (Recall that $\F$ is up to scaling orthogonal, so $O_N^+$ is invariant under $\F$.) Let us look on an example first.

\begin{ex}[$n=3$]
The most important intertwiner seems to be the adjacency matrix of the graph as only this carries the data of the graph itself. A straightforward computation gives us a diagonal matrix
$$\hat A:=\F^{-1}A\F=\Lnmatrix{
3& & & & & & & &(0,0,0)\cr\noalign{\hrule}
 &1& & & & & & &(1,0,0)\cr
 & &1& & & & & &(0,1,0)\cr
 & & &1& & & & &(0,0,1)\cr\noalign{\hrule}
 & & & &-1& & & &(0,1,1)\cr
 & & & & &-1& & &(1,0,1)\cr
 & & & & & &-1& &(1,1,0)\cr\noalign{\hrule}
 & & & & & & &-3&(1,1,1)\cr
}$$
\end{ex}

Writing some explicit matrices for $T^{(N)}_{\mergepart}$ would be slightly complicated, so we will get back to this tensor later. So, let us now do the computations for general $n$. For the adjacency matrix, we have
\begin{align*}
[\F^{-1}A\F]^\mu_\nu&={1\over 2^n}\sum_{\alpha,\beta}(-1)^{\alpha\cdot\mu}A^\alpha_\beta(-1)^{\beta\cdot\nu}={1\over 2^n}\sum_\alpha\sum_{i=1}^n(-1)^{\alpha\cdot\mu}(-1)^{(\alpha+\epsilon_i)\cdot\nu}\\
&={1\over 2^n}\left(\sum_\alpha (-1)^{\alpha\cdot(\mu+\nu)}\right)\left(\sum_{i=1}^n(-1)^{\nu_i}\right)=\delta_{\mu\nu}(n-2\deg\nu),
\end{align*}
where $\deg\nu$ is the number of \emph{ones} in $\nu$, which we will subsequently call the {\em degree} of $\nu$. So, we can say that the Fourier image of the adjacency matrix is a direct sum of identities with some scalar factors
\begin{equation}
\label{eq.hatA}
\hat A:=\F^{-1}A\F=\bigoplus_{k=0}^n(n-2k)I_{n\choose k}
\end{equation}

Imposing $\hat A$ to be an intertwiner is equivalent to saying that every subspace of $l^2(Q_n)$ consisting of elements with a given degree $d$ (w.r.t.\ the basis $(\tau_\mu)$) is invariant. So, let us denote the invariant subspaces by
$$V_k:=\{f\in l^2(Q_n)\mid \deg f=k\}.$$
Note that $\dim V_k={n\choose k}$. Note also that those spaces do not define a grading, but only a filtration on the algebra $C(Q_n)$.


So, denote by $\hat u:=\F^{-1}u\F$ the Fourier transform of the fundamental representation of the quantum automorphism group of the hypercube and by $\hat u=\hat u^{(0)}\oplus\hat u^{(1)}\oplus\cdots\oplus\hat u^{(n)}$ the decomposition according to the invariant subspaces $V_k$.

This was only how the adjacency matrix $A$ transform under $\F$. But we also need to transform the intertwiners defining the free symmetric quantum group $S_N^+$. Let us consider the tensor $T^{(N)}_{\connecterpart}\colon \C^N\otimes\C^N\to\C^N\otimes\C^n$ defined by $[T^{(N)}_{\connecterpart}]^{\alpha\beta}_{\gamma\delta}=\delta_{\alpha\beta\gamma\delta}$. (This is indeed an intertwiner of $S_N^+$ since $T^{(N)}_{\connecterpart}=T_{\mergepart}^{(N)\,*}T^{(N)}_{\mergepart}$.)

Thanks to the fact that $u$ must be block diagonal (as we just derived), we can study such intertwiners restricted to such blocks. So, denote by $P_k$ the (Fourier transform composed with the) orthogonal projection $l^2(Q_n)\to V_k$ and define the tensor $\hat T^{(N)}_{\connecterone}:=(P\otimes P)^{-1}T^{(N)}_{\connecterpart}(P\otimes P)$, which must be an intertwiner of $\hat u^{(1)}$. We can compute its entries:

\begin{align*}
[\hat T^{(N)}_{\connecterone}]^{ij}_{kl}
&=[(\F\otimes \F)^{-1}T^{(N)}_{\connecterpart}(\F\otimes \F)]^{\epsilon_i,\epsilon_j}_{\epsilon_k,\epsilon_l}
={1\over 2^{2n}}\sum_{\alpha,\beta,\gamma,\delta}(-1)^{\alpha_i+\beta_j+\gamma_k+\delta_l}\delta_{\alpha,\beta,\gamma,\delta}\\
&={1\over 2^{2n}}\sum_{\alpha}(-1)^{\alpha_i+\alpha_j+\alpha_k+\alpha_l}=
\begin{cases}
	1/2^n&\raise.5ex\vtop{\hsize=11em\parindent=0pt\raggedright if in the tuple $(i,j,k,l)$, two and two indices are the same\vrule width0pt depth 10pt}\\
	0&\text{otherwise}
\end{cases}
\end{align*}

One can write down this result also using the partition notation as
$$2^n\,\hat T^{(N)}_{\connecterone}=T^{(n)}_{\Paabb}+T^{(n)}_{\Pabba}+T^{(n)}_{\Pabab}-2T^{(n)}_{\Paaaa}$$

This is well known to be the intertwiner defining the quantum group $O_n^{-1}$. (See e.g.~\cite[Section~7]{GWgen}). Now, we may already see, where is this heading towards.

As a side remark, note that one might also want to directly compute the projection of $T^{(N)}_{\mergepart}$ to the subspace $V_1$. This is of course possible by doing a very similar computation. However, this intertwiner turns out to be equal to zero (since one can never ``pair'' a triple of indices).

\subsection{Quantum automorphism group of the hypercube}

In this section, we prove that the quantum automorphism group of the hypercube graph is the $O_n^{-1}$ -- a result originally obtained in \cite[Theorem~4.2]{BBC07}.

\begin{thm}
\label{T.hypercube}
The quantum automorphism group of the $n$-dimensional hypercube graph is the anticommutative orthogonal quantum group $O_n^{-1}$ with the representation
$$\F\big(1\oplus v\oplus(v\brewedge v)\oplus\cdots\oplus v^{\brewedge n}\big)\F^{-1},$$
where $v$ is the standard fundamental representation of $O_n^{-1}$.
\end{thm}
\begin{proof}
Denote by $u$ the fundamental representation of the quantum automorphism group of the hypercube $\Aut^+Q_n$. Denote by $\hat u:=\F^{-1}u\F$ the Fourier transform of $u$. We prove the theorem by a series of lemmata. We start with results derived in the previous section.

\begin{lem}
The representation $\hat u$ decomposes as $\hat u=\hat u^{(0)}\oplus\hat u^{(1)}\oplus\cdots\oplus\hat u^{(n)}$.
\end{lem}
\begin{proof}
Follows from the form of the Fourier transform of the adjacency matrix \eqref{eq.hatA}.
\end{proof}

\begin{lem}
The subrepresentation $\hat u^{(1)}$ satisfies the defining relation for $O_n^{-1}$.
\end{lem}
\begin{proof}
Follows from $\hat T^{(N)}_{\connecterone}$ being an intertwiner of $\hat u^{(1)}$.
\end{proof}

\begin{lem}
The subrepresentation $\hat u^{(1)}$ is a faithful representation of $\Aut^+Q_n$.
\end{lem}
\begin{proof}
The representation $\hat u$ acts on $Q_n$ through the coaction $\tau_\mu\mapsto\sum_\nu\tau_\nu\otimes u^\nu_\mu$. Since the algebra $C(Q_n)$ is generated by the invariant subspace $V_1=\spanlin\{\tau_i\}_{i=1}^n$, this coaction is uniquely determined by the coaction of $\hat u^{(1)}$ on this invariant subspace. In particular, the entries of $\hat u$ must be generated by the entries of $\hat u^{(1)}$.
\end{proof}

Consequently, $\Aut^+Q_n$ is a quantum subgroup of $O_n^{-1}$. Now it remains to prove the opposite inclusion.

\begin{lem}
\label{L.Qncoact}
The mapping $\tau_i\mapsto\sum_{j=1}^n\tau_j\otimes v^j_i$ extends to a $*$-homomorphism $\alpha\colon C(Q_n)\to C(Q_n)\otimes \Olg(O_n^{-1})$. The subspaces $V_l$ are invariant subspaces of this action for every $l=0,\dots,n$.
\end{lem}
\begin{proof}
We have to check that the images of the generators $\tau_i$ satisfy the generating relations. That is:
$$\alpha(\tau_i)^*=\alpha(\tau_i),\qquad\alpha(\tau_i)^2=1,\qquad\alpha(\tau_i)\alpha(\tau_j)=\alpha(\tau_j)\alpha(\tau_i).$$
The first one is obvious as both $\tau_i$ and $v^j_i$ are self-adjoint. For the second, we have $\alpha(\tau_i)^2=\sum_{j,k}\tau_j\tau_k\otimes u^j_iu^k_i$. While $\tau_j$ commutes with $\tau_k$, we have that $u^j_i$ anticommutes with $u^k_i$ unless $j=k$. So, the entries for $j\neq k$ subtract to zero and we are left with $\sum_j \tau_j\tau_j\otimes u^j_iu^j_i=1$. Finally, the last one (assuming $i\neq j$) is again easy as we have $\alpha(\tau_i)\alpha(\tau_j)=\sum_{k,l}\tau_k\tau_l\otimes u^k_iu^l_j$, where everything commutes (unless $k=l$, but for those summands we have $\sum_k \tau_k\tau_k\otimes u^k_iu^k_j=0$).

It remains to show that $V_l$ are invariant subspaces. Take arbitrary element $\tau_{i_1}\cdots\tau_{i_l}\in V_l$ (assuming $i_1<\cdots<i_l$) and write explicitly
\begin{equation}\label{eq.Onact}
\alpha(\tau_{i_1}\cdots\tau_{i_l})=\sum_{j_1,\dots,j_l=1}^n\tau_{j_1}\cdots\tau_{j_l}\otimes u^{i_1}_{j_1}\cdots u^{i_l}_{j_l}.
\end{equation}
In the cases, where $j_a=j_b$ for some $a,b$, the contribution of the sum must equal to zero since the corresponding $\tau_{j_a}$ and $\tau_{j_b}$ commute, while $u^{i_a}_{j_a}$ and $u^{i_b}_{j_b}$ anticommute. So, we are actually summing over elements $\tau_{j_1}\cdots\tau_{j_l}\in V_l$ only, which is what we wanted to prove.
\end{proof}

This means that $O_n^{-1}$ is a quantum subgroup of $\Aut^+Q_n$, which finishes the proof that $\Aut^+Q_n\simeq O_n^{-1}$. Finally, from the explicit expression \eqref{eq.Onact}, it is clear that $O_n^{-1}$ acts on the invariant subspaces $V_l$ indeed through the representations $\hat u^{(l)}=v^{\brewedge l}$.
\end{proof}

\section{Cayley graphs of abelian groups: General picture}
\label{sec.Cayley}

The fact that the Fourier transform diagonalizes the adjacency matrix of the hypercube graph is not a coincidence. In the graph theory, it is a well known fact, which holds for any Cayley graph of an abelian group. (See \cite{LZ19} for a nice survey on the spectral theory of Cayley graphs.)

Let $\Gamma$ be a finite abelian group. We denote by $\Irr\Gamma\subset C(\Gamma)$ the set of all irreducible characters (that is, one-dimensional representations; since $\Gamma$ is abelian, all irreducible representations are in fact one-dimensional). Note that $\Irr\Gamma$ forms a basis of $C(\Gamma)$ and expressing a function in this basis is exactly the \emph{Fourier transform} on $\Gamma$.

Let $S\subset\Gamma$ be a set of generators of $\Gamma$. As in the last section, we are going to denote the elements of $\Gamma$ by Greek letters and the operation on $\Gamma$ as addition. The \emph{Cayley graph} of the group $\Gamma$ with respect to the generating set $S$ denoted by $\Cay(\Gamma,S)$ is a directed graph defined on the vertex set $\Gamma$ with an edge $(\alpha,\beta)$ for every pair of elements such that $\beta=\alpha+\theta$ for some $\theta\in S$. If $S$ is closed under the group inversion, then the Cayley graph is actually undirected (for every edge, one also has the opposite one). So, the adjacency matrix of $\Cay(\Gamma,S)$ is of the form
$$[A]^\beta_\alpha=
\begin{cases}
1&\text{if $\beta=\alpha+\theta$ for some $\theta\in S$,}\\
0&\text{otherwise.}
\end{cases}
$$

\begin{prop}[\cite{Lov75,Bab79}]
\label{P.lambda}
Let $\Gamma$ be a finite abelian group and $S$ its generating set. Denote by $A$ the adjacency matrix associated to the Cayley graph $\Cay(\Gamma,S)$. Then $\Irr\Gamma$ forms the eigenbasis of $A$. Given $\chi\in\Irr\Gamma$, its eigenvalue is given by
\begin{equation}
\label{eq.lambda}
\lambda_\chi=\sum_{\theta\in S}\chi(-\theta).
\end{equation}
\end{prop}
\begin{proof}
This is just a straightforward check. Take any $\chi\in\Irr\Gamma$ then we have
\[[A\chi]^\alpha=\sum_{\beta\in\Gamma}A^\alpha_\beta\chi(\beta)=\sum_{\theta\in S}\chi(\alpha-\theta)=\sum_{\theta\in S}\chi(-\theta)\chi(\alpha)=\lambda_\chi\chi(\alpha).\qedhere\]
\end{proof}

This result suggests the strategy for determining the quantum automorphism group for any Cayley graph corresponding to an abelian group: express everything in the basis of irreducible characters. Since this diagonalizes the adjacency matrix, the meaning of it as an intertwiner becomes obvious. On the other hand, it might be slightly more complicated to discover the meaning of the intertwiner $T^{(N)}_{\mergepart}$.

\begin{alg}[Determining $\Aut^+\Cay(\Gamma,S)$]
\label{Algorithm}
Let $\Gamma$ be an abelian group and $S$ its generating set. We are trying to determine the quantum automorphism group $\Aut^+\Cay(\Gamma,S)$ with fundamental representation $u$. In order to do so, we perform the following steps:
\begin{enumerate}
\item Determine the irreducible characters of $\Gamma$. Suppose that $\Gamma=\Z_{m_1}\times\cdots\times\Z_{m_n}$. Then $\Irr\Gamma=\{\tau_\mu\}_{\mu\in\Gamma}$ with $\tau_\mu(\alpha)=\prod_{i=1}^n\gamma_i^{\alpha_i\mu_i}$, where $\gamma_i$ is some primitive $m_i$-th root of unity for every $i$.
\item Determine the spectrum of $A_\Gamma$ using Equation~\eqref{eq.lambda}.
\item Denoting by $\lambda_0>\lambda_1>\dots>\lambda_d$ the mutually distinct eigenvalues of $A_\Gamma$, determine the corresponding eigenspaces $V_0,V_1,\dots,V_d$. (Note that we will always have $V_0=\spanlin\{\tau_0\}$, where 0 is the group identity.)
\item The eigenspaces are invariant subspaces of $u$. To formulate it slightly differently: We may define the \emph{Fourier transform} $\F$ as the matrix corresponding to the change of basis $(\delta_\alpha)\mapsto(\tau_\mu)$, that is $\F^\alpha_\mu=\tau_\mu(\alpha)$. Then we can define $\hat u=\F^{-1}u\F$, which decomposes as a direct sum $\hat u=\hat u^{(0)}\oplus\hat u^{(1)}\oplus\cdots\oplus\hat u^{(d)}$.
\item Choose some of the spaces and define $W:=V_{i_1}\oplus V_{i_2}\oplus\cdots$ in such a way that $W$ generates $C(\Gamma)$ as an algebra. (In our examples below, it will be enough to take $W:=V_1$, but it does not always have to be like this.) This means that $v:=\hat u^{(i_1)}\oplus\hat u^{(i_2)}\oplus\cdots$ is a faithful representation of $\Aut^+\Cay(\Gamma,S)$. (Since the coaction of $u$ or $\hat u$ restricted to $W$ must then again uniquely extend to the whole space $C(X)$ and hence we can recover the whole $u$ this way.)
\item Any non-crossing partition $p\in NC(k,l)$ defines an intertwiner $T_p^{(N)}\in\Mor(u^{\otimes k},u^{\otimes l})$, where $N=|\Gamma|$. We define $\hat T_p^{(N)}:=\F^{-1\,\otimes l}T_p^{(N)}\F^{\otimes k}$, which has to be an intertwiner $\hat T_p^{(N)}\in\Mor(\hat u^{\otimes k},\hat u^{\otimes l})$. It is actually enough to study the intertwiners corresponding to the block partitions $b_{k,l}\in\Part(k,l)$ -- partitions, where all the $k$ upper and $l$ lower points are in a single block, so $[T_{b_{k,l}}^{(N)}]^{\beta_1,\dots,\beta_l}_{\alpha_1,\dots,\alpha_k}=\delta_{\alpha_1,\dots,\alpha_k,\beta_1,\dots,\beta_l}$. The entries of the Fourier transformed intertwiner are easily computed as 
\begin{equation}\label{eq.That}
[\hat T_{b_{k,l}}]^{\nu_1,\dots,\nu_l}_{\mu_1,\dots,\mu_k}=N^{1-l}\delta_{\mu_1+\cdots+\mu_k,\nu_1+\cdots+\nu_l}.
\end{equation}
In particular, we can focus on the subspace $W$ and study the relations $\hat T^{(W)}_pv^{\otimes k}=v^{\otimes l}\hat T^{(W)}_p$, where $\hat T^{(W)}_p=U^{\otimes l}\hat T_p^{(N)}U^{*\otimes k}$, where $U$ is the coisometry $V_0\oplus\cdots\oplus V_d\to W$.
\end{enumerate}
\end{alg}

\section{Halved hypercube}
\label{sec.halfQ}

The hypercube graph is bipartite and hence we can create a new graph of it by the procedure of \emph{halving} -- taking one of the two components of the associated \emph{distance-two-graph}. Taking a natural number $n\in\N$, we define the $(n+1)$-dimensional \emph{halved hypercube graph} $\frac{1}{2}Q_{n+1}$ obtained by halving the ordinary hypercube $Q_{n+1}$. That is, we take all the even vertices in $Q_{n+1}$ (equivalently, all odd vertices) and connect by an edge every pair of vertices that were in the distance two in the original hypercube $Q_{n+1}$.

There is also a simpler definition of $\frac{1}{2}Q_{n+1}$. Take the $n$-dimensional hypercube $Q_n$ and add an additional edge for every pair of vertices in distance two. This is also known as \emph{squaring} the graph. It holds that $Q_n^2=\frac{1}{2}Q_{n+1}$. Using this description, we can write the adjacency matrix as follows
$$[A]^\alpha_\beta=
\begin{cases}
1&\text{if $\beta=\alpha+\epsilon_i$ for some $i$}\\
1&\text{$\beta=\alpha+\epsilon_i+\epsilon_j$ for some $i\neq j$}\\
0&\text{otherwise}
\end{cases}$$

Consequently, we see that $\frac{1}{2}Q_{n+1}=Q_n^2$ is a Cayley graph corresponding to the group $\Gamma=\Z_2^n$ with respect to the generating set $S=\{\epsilon_i\}_{i=1}^n\cup\{\epsilon_i+\epsilon_j\}_{i<j}$. In particular, the number of vertices is $N:=2^n$.

Now we would like to determine the quantum automorphism group of the halved hypercube graph $\frac{1}{2}Q_{n+1}$. Let us first summarize some known results. For $n\le 2$, the graph is actually the full graph on $N=2^n$ vertices, so the quantum automorphism group is the free symmetric quantum group $S_N^+$ (for $n=0,1$, i.e.\ $N=1,2$, this actually coincides with the classical one $S_N$). For $n=3$, the graph is the complement of the graph consisting of four isolated segments. Hence, its quantum automorphism group is the free hyperoctahedral quantum group $H_4^+=\Z_2\wr_*S_4^+$. (Here $\wr_*$ denotes the \emph{free wreath product}, which describes the quantum automorphism group of $n$ copies of a given graph \cite{Bic04}.)

So, the question is what is the quantum automorphism group of $\frac{1}{2}Q_{n+1}$ for $n>3$. The classical automorphism group is known to be $\Z_2^n\rtimes S_{n+1}$. More precisely, it is the index two subgroup of the hyperoctahedral group $H_{n+1}=\Z_2\wr S_{n+1}=\Z_2^{n+1}\rtimes S_{n+1}$ (the symmetry group of the hypercube) imposing that the product of all the $\Z_2$-signs is equal to one. This group is also known as the Coxeter group of type $D$. Since the quantum automorphism group of $Q_{n+1}$ is $O_{n+1}^{-1}$, we may expect that the answer for the halved hypercube $\frac{1}{2}Q_{n+1}$ should be the anticommutative special orthogonal group $SO_{n+1}^{-1}$.

\subsection{Determining $\Aut^+\frac{1}{2}Q_{n+1}$}
We follow Algorithm~\ref{Algorithm}. We start by computing the spectrum using Equation \eqref{eq.lambda}:
$$\lambda_\mu=\sum_{i=1}^n\tau_\mu(\epsilon_i)+\sum_{i<j}\tau_\mu(\epsilon_i+\epsilon_j)=\sum_i(-1)^{\mu_i}+\sum_{i<j}(-1)^{\mu_i+\mu_j}=l_\mu+\frac{1}{2}l_\mu^2-\frac{n}{2},$$
where $l_\mu=\sum_i(-1)^{\mu_i}=n-2\deg\mu$. Note that the eigenvalue depends again only on the degree of $\mu$ (the number of non-zero entries). So, denote $\lambda_d:=\lambda_\mu$ for $d=\deg\mu$. So, $\lambda_d=\frac{1}{2}(l_d^2+2l_d-n)$ with $l_d=n-2d$. After some computation, one can find out that
$$\lambda_d=\frac{1}{2}\left((2d-n-1)^2-n-1\right)$$
In contrast with the computation for the hypercube $Q_n$, the eigenvalues $\lambda_0,\dots,\lambda_n$ are not distinct. Instead, $\lambda_d=\lambda_{n+1-d}$. Consequently, the matrix $\hat A$ as an intertwiner does not imply that the subspaces $V_i:=\spanlin\{\tau_\mu\mid\deg\mu=i\}$ are invariant. Instead, we have the following invariant subspaces: $\tilde V_0:=V_0$, $\tilde V_1:=V_1+ V_n$, $\tilde V_2:=V_2+ V_{n+1}$ and so on. In general, $\tilde V_i=V_i+ V_{n+1-i}$ is an invariant subspace for every $i=0,\dots,\lfloor\frac{n+1}{2}\rfloor$ (using the convention $V_{n+1}=\{0\}$).

In order to describe the invariant spaces, define $\tau_{n+1}:=\tau_1\cdots\tau_n$. Then $\tau_1,\dots,\tau_n$, $\tau_{n+1}$ is the basis of $\tilde V_1$. The basis of each $\tilde V_i$ is exactly the set $\{\tau_\alpha\}$ with $\alpha\in\Z_2^{n+1}$, $\deg\alpha=i$. Denote by $u$ the fundamental representation of $\Aut^+\frac{1}{2}Q_{n+1}$ and by $\hat u^{(i)}$ the block of $\hat u=\F^{-1}u\F$ corresponding to the invariant subspace $\tilde V_i$.

Obviously, $\tau_1,\dots,\tau_{n+1}$ are also generators of $C(\frac{1}{2}Q_{n+1})=C(\Z_2^{n})$ (since already $\tau_1,\dots,\tau_n$ are generators) and we can write the algebra by generators and relations as
$$C({\textstyle\frac{1}{2}}Q_{n+1})=C^*(\tau_1,\dots,\tau_{n+1}\mid \tau_i^2=1,\;\tau_i\tau_j=\tau_j\tau_i,\;\tau_i^*=\tau_i,\;\tau_1\cdots\tau_{n+1}=1).$$

\begin{thm}
\label{T.demicube}
Consider $n\in\N\setminus\{1,3\}$. The quantum automorphism group of the halved hypercube graph $\frac{1}{2}Q_{n+1}$ is isomorphic to the anticommutative special orthogonal group $SO_{n+1}^{-1}$. It acts through the fundamental representation $u$  with $\hat u^{(i)}=v^{\brewedge i}$, where $v$ is the fundamental representation of $SO_{n+1}^{-1}$.
\end{thm}
\begin{proof}
We follow the proof of Theorem~\ref{T.hypercube}. Let us first prove that $SO_{n+1}^{-1}\subset\Aut^+\frac{1}{2}Q_{n+1}$, that is, $SO_{n+1}^{-1}$ really acts on the halved hypercube. To do this, we are going to show that the mapping
$$\tau_j\mapsto\sum_{i=1}^{n+1}\tau_i\otimes u^i_j$$
extends to a $*$-homomorphism. Most of the work was already done in Lemma~\ref{L.Qncoact}. It only remains to prove that the extra relation we have here $\tau_1\cdots\tau_{n+1}=1$ is also preserved under this action. That is, we need to show that
$$1=\sum_{i_1,\dots,i_{n+1}=1}^{n+1}\tau_{i_1}\cdots\tau_{i_{n+1}}\otimes u_{i_11}\cdots u_{i_{n+1},n+1}$$
This is indeed true thanks to the anticommutative determinant relation $\adet u=\nobreak 1$.

Now for the converse direction, consider again the intertwiner $\hat T^{(N)}_{\connecterpart}$ and compute its restriction to $\tilde V_1$. It is easy to check that we obtain the same formula
$$
[\hat T^{(N)}_{\connecterpart}]^{ij}_{kl}=
\begin{cases}
	1/2^n&\raise.5ex\vtop{\hsize=11em\parindent=0pt\raggedright if in the tuple $(i,j,k,l)$, two and two indices are the same\vrule width0pt depth 10pt}\\
	0&\text{otherwise}
\end{cases}
$$
even on the ``extended'' space $\tilde V_1=V_1\oplus V_n$ (this is the point, where the assumption $n\neq 1,3$ is needed). Consequently, we have proven that $O_{n+1}^{-1}\supset\Aut^+\frac{1}{2}Q_{n+1}$.

Finally, it remains to find some intertwiner that would push us to the $SO_{n+1}^{-1}$. Consider the block partition $b_{n+1}=\blockpart\in\Part(n+1,0)$. The corresponding intertwiner is then of the form $[T^{(N)}_{\blockpart}]_{\alpha_1,\dots,\alpha_{n+1}}=\delta_{\alpha_1,\dots,\alpha_{n+1}}$. We are interested in restriction of its Fourier transform on $\tilde V_1$ (more precisely, $\tilde V_1^{\otimes (n+1)}$):
$$[\hat T^{(N)}_{\blockpart}]_{i_1,\dots,i_{n+1}}=
\begin{cases}
1&\text{if $(i_1,\dots,i_{n+1})$ is a permutation of $(1,\dots,n+1)$}\\
0&\text{otherwise}
\end{cases}$$
This is exactly the antisymmetrizer $\breve\A_n\in\Mor((\hat u^{(1)})^{\otimes n},1)$, which exactly corresponds to the relation $\adet \hat u^{(1)}=1$.
\end{proof}

\subsection{Open problem}

Let us finish this section with links to some open questions and related research. Note that there are free quantum analogues of the Coxeter groups of series $A$ (that is, the symmetric groups $S_n$) and series $B$ or $C$ (that is, the hyperoctahedral groups $H_n$), but so far we do not have any really free analogue of the Coxeter series $D$ (recall that series $D$ consists exactly of the symmetry groups of halved hypercubes $\frac{1}{2}Q_n$). The series of the anticommutative special orthogonal groups $SO_n^{-1}$ \emph{is} a liberation of the Coxeter groups of type $D$, but we should not call them \emph{free} as they obey some sort of commutativity laws (namely the anticommutativity).

\begin{quest}
Is there a free analogue for the Coxeter groups of series $D$?
\end{quest}

One particular candidate was recently discovered in~\cite{GD4} for $n=4$. For general $n$, the question is still open. If some candidates appear, then the natural follow-up question would be: Is there some graph, whose quantum symmetry is this?

\section{Folded hypercube}
\label{sec.FQ}

Folded hypercube is another graph that can be derived from the hypercube graph. Consider again $n\in\N$. The \emph{$(n+1)$-dimensional folded hypercube graph} $FQ_{n+1}$ is a quotient of the hypercube $Q_{n+1}$ obtained by identifying the opposite corners, so $\alpha\equiv\alpha+\iota$, where $\iota=(1,\dots,1)=\epsilon_1+\cdots+\epsilon_n$. By this, we end up with a graph having only half of the vertices, that is, $N=2^n$.

Also here, there is a more convenient description. The $(n+1)$-dimensional folded hypercube can be obtained from the $n$-dimensional ordinary hypercube by connecting all the opposite corners by an additional edge. So, the adjacency matrix can be written as
$$[A]^\alpha_\beta=
\begin{cases}
1&\text{if $\beta=\alpha+\epsilon_i$ for some $i$,}\\
1&\text{if $\beta=\alpha+\iota$,}\\
0&\text{otherwise.}
\end{cases}$$
In other words, it is the Cayley graph of $\Z_2^n$ with respect to the generating set $\{\epsilon_1,\dots,\epsilon_n,\iota\}$.

Now, let us again review, what is known about its classical and quantum symmetries. For $n\le 2$, the folded hypercube $FQ_{n+1}$ is just the complete graph on $N=2^n$ vertices, so its quantum automorphism group is $S_N^+$. For $n=3$, it is the complete bipartite graph on $4+4$ vertices, which is the complement of the disjoint union $K_4\sqcup K_4$, so its quantum automorphism group is $S_4^+\wr_*\Z_2$ (again, we refer to \cite{Bic04} for explanation of the free wreath product). So, the interesting area is $n>3$.

The classical automorphism group of $FQ_{n+1}$ for $n>3$ is of the form $\Z_2^{n}\rtimes S_{n+1}$ \cite{Mir16}, but it is a different semidirect product than in the case of the halved cube. Here, we take the quotient group of $H_{n+1}=\Z_2^{n+1}\rtimes S_{n+1}$ by identifying the $(n+1)$-tuple of signs with the opposite ones. In other words, it is the projective version $PH_{n+1}$. Therefore, we expect $PO_{n+1}^{-1}$ to be the quantum automorphism group.

In \cite{Sch20fq}, it was proven for $n$ even that the quantum automorphism group is actually $SO_{n+1}^{-1}$, which matches our guess since it is isomorphic with $PO_{n+1}^{-1}$ according to Proposition~\ref{P.POiso}.

\subsection{Determining $\Aut^+FQ_{n+1}$}
Let us now compute the eigenvalues of the adjacency matrix:
\begin{align*}
\lambda_\mu&=\sum_{i=1}^n\tau_\mu(\epsilon_i)+\tau_\mu(\iota)=\sum_{i=1}^n(-1)^{\mu_i}+(-1)^{\mu_1+\cdots+\mu_n}\\
&=n-2\deg\mu+(-1)^{\deg\mu}=n+1-4\lceil1/2\,\deg\mu\rceil
\end{align*}

Again, the eigenvalue depends only on the degree of $\mu$, but again the values of $\lambda_d:=\lambda_\mu$ with $\deg\mu=d$ are not mutually distinct. In this case, we have $\lambda_{2d-1}=\lambda_{2d}$. So, we have invariant subspaces $\tilde V_0:=V_0$, $\tilde V_2:=V_1\oplus V_2$ and so on. That is, $\tilde V_{2i}=V_{2i-1}\oplus V_{2i}$, $i=0,1,\dots,\lceil n/2\rceil$ (with $\tilde V_{n+1}=V_n$ if $n$ is odd). Denoting by $u$ the fundamental representation of $\Aut^+FQ_{n+1}$, we denote by
$$\hat u=\hat u^{(0)}\oplus\hat u^{(2)}\oplus u^{(4)}\oplus\cdots$$
the decomposition of its Fourier transform according to the invariant subspaces.

Denote $\tau_{n+1}:=1$, so the elements $\tau_{ij}:=\tau_{ji}:=\tau_i\tau_j$ with $1\le i<j\le n+1$ form a basis of $\tilde V_2$. In general, the basis of $\tilde V_{2i}$ is $(\tau_\alpha)$ with $\alpha\in\Z_2^{n+1}, \deg\alpha=2i$. We can use the basis of $\tilde V_2$ as a generating set of $C(\Z_2^n)=C(FQ_{n+1})$:
$$C(FQ_{n+1})=C^*\left(\tau_{ij},\;1\le i<j\le n+1\middle|\begin{matrix} \tau_{ij}^2=1,\;\tau_{ij}^*=\tau_{ij},\\\tau_{ij}\tau_{kl}=\tau_{kl}\tau_{ij},\;\tau_{ij}\tau_{ik}=\tau_{jk}\end{matrix}\right).$$

Alternatively, one can view $C(FQ_{n+1})$ as the subalgebra of $C(\Z_2^{n+1})=C(Q_{n+1})$
$$C(Q_{n+1})=C^*(\tau_1,\dots,\tau_{n+1}\mid \tau_i^2=1,\;\tau_i^*=\tau_i,\;\tau_i\tau_j=\tau_j\tau_i)$$
generated by the elements $\tau_{ij}=\tau_{ji}=\tau_i\tau_j$. This exactly corresponds to the fact that $FQ_{n+1}$ is a quotient graph of $Q_{n+1}$.

\begin{thm}
\label{T.fcube}
Consider $n\in\N\setminus\{1,3\}$. The quantum automorphism group of the folded hypercube graph $FQ_{n+1}$ is isomorphic to the anticommutative projective orthogonal group $PO_{n+1}^{-1}$. It acts through the fundamental representation $u$  with $\hat u^{(i)}=v^{\brewedge 2i}$, where $v$ is the fundamental representation of $O_{n+1}^{-1}$.
\end{thm}

Before proving this theorem, we need to do some preparatory work first. At this point, it is not even clear whether the prescribed representation $\bigoplus v^{\brewedge 2i}$ is a faithful representation of $PO_{n+1}^{-1}$. We are actually going to prove that $v\brewedge v$ is a faithful representation. As a second step, we need to characterize $\Olg(PO_{n+1}^{-1})$ by generators and relations in terms of this representation $v\brewedge v$. Equivalently, we need to find suitable generators of the representation category $\RCat$ associated to $v\brewedge v$. Only this allows us to use our standard machinery and prove Theorem~\ref{T.fcube}.

This preparatory work is done in Subsection~\ref{secc.projwedge}. The proof of Theorem~\ref{T.fcube} itself is then formulated in Subsection~\ref{secc.fcubeproof}.

\subsection{Projective version represented by exterior product}
\label{secc.projwedge}

In Section~\ref{secc.proj}, we defined the projective version of a compact matrix quantum group again as a compact matrix quantum group with the fundamental representation of the form $u\otimes u$. In the classical case, we have also another faithful representation at our disposal.

\begin{prop}
\label{P.projwedge}
Consider $G\subset O_n$, $n>1$. Denote by $u$ its fundamental representation. Then $u\wedge u$ is a faithful representation of $PG$.
\end{prop}

Before formulating the proof, let us clarify the definition of projective groups. Our definition from Section~\ref{secc.proj} works for orthogonal groups only. For a general matrix group $G\subset GL_n$, one typically defines its projective version to be the quotient $PG=G/\lambda I$, $\lambda\in\C\setminus\{0\}$ (thus, in particular, $PGL_n=GL_n/Z(GL_n)$). Note that if $G$ is orthogonal, then $\lambda I\in G$ only for $\lambda=\pm 1$. Therefore, assuming $G\subset O_n$, our definition $PG=\{A\otimes A\mid A\in G\}$ is compatible with the general one (since we can reconstruct $A$ from $A\otimes A$ up to a global sign).

Finally, let us mention that over $\C$, we obviously have $PGL_n=PSL_n$, which is known to be a simple group, that is, it has no non-trivial normal subgroups.

\begin{proof}\!\!\!\footnote{Credit for this proof goes to a math.StackExchange user Joshua Mundinger \url{https://math.stackexchange.com/a/4049085/359512}}
Denote $V:=\C^n$, so that $GL_n=GL(V)$. Consider the homomorphism $\phi\colon GL(V)\to GL(V\wedge V)$ mapping $A\mapsto A\wedge A$. Since it maps multiples of identity to multiples of identity, it induces a homomorphism $\tilde\phi\colon PGL(V)\to PGL(V\wedge V)$. Since the $PGL(V)$ is simple (and the mapping $\tilde\phi$ is obviously non-trivial), we must have that $\tilde\phi$ is injective. Consequently, the kernel of $\phi$ is contained in scalar matrices.

Now, we can restrict $\phi$ to any subgroup $G\subset GL_n$. Note that we can factor $\phi$ as $\phi\colon G\overset{\phi_1}{\to} G'\overset{\phi_2}{\to} G''$, where $G'=\{A\otimes A\mid A\in G\}$ and $G''=\{A\wedge A\mid A\in G\}$. We need to prove that $\phi_2\colon G'\to G''$ is an isomorphism for $G\subset O_n$ as we have $G'\simeq PG$ in this case. Recall that we have the property $\ker\phi\subset\{\lambda I\}$. If $G\subset O_n$, we must actually have $\ker\phi\subset\{\pm I\}$. But we know that $\ker\phi_1=\{\pm I\}$ and hence $\phi_2$ must indeed be an isomorphism.
\end{proof}

It is known that the anticommutative orthogonal group can be obtained from the ordinary one by a 2-cocycle twist. Consequently, the two quantum groups are \emph{monoidally equivalent}. More precisely, there exists a monoidal isomorphism of the corresponding representation categories $\Rep O_n\to\Rep O_n^{-1}$ mapping the fundamental representations one to another (that is, there is also a monoidal isomorphism $\RCat_{O_n}\to\RCat_{O_n^{-1}}$). See e.g.\ \cite{BBC07,GWgen}. This implies the following corollary.

\begin{cor}
\label{P.antiprojwedge}
Denote by $\breve u$ the fundamental representation of $O_n^{-1}$, $n>1$. Then $\breve u\brewedge\breve u$ is a faithful representation of $PO_n^{-1}$.
\end{cor}
\begin{proof}
Denote by $u$ the fundamental representation of $O_n$. As mentioned above, there is a monoidal equivalence $\Rep O_n\to\Rep O_n^{-1}$ mapping $u\mapsto\breve u$. It is easy to check that this monoidal equivalence also maps $\A_2^*\A_2\in\Mor(u\otimes u,u\otimes u)$ to $\breve\A_2^*\breve\A_2\in\Mor(\breve u\otimes\breve u,\breve u\otimes\breve u)$. Consequently, it must map $u\wedge u$ to $\breve u\brewedge\breve u$. Since the former is a faithful representation of $PO_n$, the latter must be a faithful representation of $PO_n^{-1}$.
\end{proof}


In the following text, we will denote the projective orthogonal group represented by the matrix $u\wedge u$ by $\hat PO_n=(\Olg(PO_n),u\wedge u)$. (The hat should remind us about the wedge product $\wedge$.) Expressing the representation category $\RCat_{\hat PO_n}$ in terms of the representation category $\RCat_{O_n}$ is easy: We only have to compose all the intertwiners with the antisymmetrizer $\A_{2}\colon x\otimes y\mapsto x\wedge y$. That is,
\begin{align*}
\RCat_{\hat PO_n}(k,l)&=\{\A_{(2)}^{\otimes l}T\A_{(2)}^{*\otimes k}\mid T\in\RCat_{O_n}(2k,2l)\}\\
                      &=\{\A_{(2)}^{\otimes l}T_p\A_{(2)}^{*\otimes k}\mid p\in \Pair_n(2k,2l)\},
\end{align*}
where $\Pair_n=\langle\pairpart,\crosspart\rangle_n=\spanlin\{p\mid\text{$p$ is a pairing}\}$. (Pairing is a partition, where all blocks have size two. By an old result of Brauer \cite{Bra37}, this is exactly the category corresponding to the orthogonal group. See also e.g.~\cite{BS09,Web17,GroThesis}.) However, the question is, what is the generating set of this category. Finding some small set of generators is actually not so easy as it may seem on the first sight.

In order to understand the following text, one needs to familiarize the category operations on linear categories of partitions (or at least the linear category of all pairings $\Pair_n$). The rough idea is that in order to perform the composition of two pairings, one should simply follow the lines and, if needed, replace all loops by the factor $n$. We refer to \cite[Section~3]{GWintsp} for more details.

In the following computations, we are going to treat the antisymmetrizer as a projection rather than a coisometry, so $\A_2^*=\A_2$. In this sense, it can be expressed in terms of partitions as $\frac{1}{2}(\Pabba-\Pabab)$. However, we are going to use a more convenient notation: In the diagrams, we will denote the antisymmetrizer by \Acol. So, for example, the antisymmetrization of $\immersepart$ will be denoted by $\PAimmerse$, the antisymmetrization of the identity is $\PAid=\frac{1}{2}(\Pabba-\Pabab)$. For any partition $p\in\Part(2k,2l)$, we are going to denote its antisymmetrization by $\mathring p=\PAid^{\otimes l}\cdot p\cdot \PAid^{\otimes k}$. Consequently, we can write
\begin{equation}\label{eq.POnCat}
\RCat_{\hat PO_n}(k,l)=\{T_{\mathring p}\mid p\in \Pair_n(2k,2l)\},
\end{equation}
so $\RCat_{\hat PO_n}$ is modelled by a diagrammatic category $\Pair_n^\circ:=\{\mathring p\mid p\in\Pair_n\}$.

\begin{thm}
\label{T.projwedge}
Consider $n\neq 2,4,6,8$. Then the category $\Pair_n^\circ$ is generated by the set $\{\PApair,\PAimmerse\}$.
\end{thm}

Before proving the theorem, let us mention two important Corollaries:

\begin{cor}
\label{C.projwedge}
Consider $n\neq 2,4,6,8$. Then the representation category $\RCat_{\hat PO_n}$ is generated by $\{T^{(n)}_{\PApair},T^{(n)}_{\PAimmerse}\}$.
\end{cor}
\begin{proof}
Follows directly from Theorem~\ref{T.projwedge} and Equation~\eqref{eq.POnCat}.
\end{proof}

\begin{cor}
\label{C.antiprojwedge}
Consider $n\neq 2,4,6,8$. Then the representation category $\RCat_{\hat PO_n^{-1}}$ is generated by $\{\breve T^{(n)}_{\PApair},\breve T^{(n)}_{\PAimmerse}\}$.
\end{cor}
\begin{proof}
As we mentioned earlier, $O_n$ is monoidally equivalent with $O_n^{-1}$, the monoidal equivalence maps the fundamental representation $u$ of $O_n$ to the fundamental representation $\breve u$ of $O_n^{-1}$. As for the intertwiners, it maps $T_p\in\Mor(u^{\otimes k},u^{\otimes l})$ for $p\in\Pair_n$ to $\breve T_p\in\Mor(\breve u^{\otimes k},\breve u^{\otimes l})$ defined at the end of Section~\ref{secc.part}. The Corollary then follows by restricting the monoidal equivalence to the full subcategory $\RCat_{PO_n}\subset\RCat_{O_n}$ and applying to Corollary~\ref{C.projwedge}.
\end{proof}

Now, we focus on the proof of Theorem~\ref{T.projwedge}. To make it easier to follow, we split it into several lemmata.

First, let us do a small remark on rotations in the category $\Pair_n^\circ$. This category (as well as the category $\RCat_{PO_n}$) is rigid and the duality morphism looks like this: $\PApair$. This again allows to do rotations in the category, but those rotations look a bit different than in the original category $\Pair_n$. Consider some $p\in\Part(2k,2l)$. Let us call each pair of some $(2i+1)$-st and $(2i+2)$-nd point on either lower or upper row in $p$ a \emph{two-point}. When drawing $\mathring p$, those two-points are highlighted by the ellipse $\Acol$. The element $\PApair$ then allows to rotate those two-points in $\mathring p$ as a whole, not separately. For instance, rotating $\PAthree$, we may obtain $\PAimmerse$, but we cannot obtain \LPartition{}{0.3:0.8,1.2;0.3:1.8,2.2;0.3:2.8,3.2}[aaa]. (The latter would actually equal to zero due to the antisymmetrization: $\LPartition{}{0.3:0.8,1.2}[a]=\frac{1}{2}(\pairpart-\pairpart)=0$.)

\begin{lem}
\label{L1.projwedge}
Consider $n\neq 2$. Then the category $\Pair_n^\circ$ is generated by the set $\{\PApair,\PAthree,\PAfour,\PAfive,\PAcross\}$.
\end{lem}

\begin{proof}
Denote by $\Cat$ the category generated by the given generators. Notice that we have the duality morphism $\PApair$ among the generators, so $\Cat$ is a rigid category and we can consider everything up to rotation now. For instance, we could equivalently consider $\PAdoubleimmerse$ instead of $\PAfive$ or $\PAimmerse$ instead of $\PAthree$ in the generating set.

As the first step, we prove that $\mathring p_k\in\Cat$ for every $k\in\N\setminus\{1\}$, where $p_k=
\Partition{
\Pblock 0to0.8:1,8
\Pblock 0to0.3:2,3
\Pblock 0to0.3:6,7
\Ptext (4.5,0.2) {$\cdots$}
}\in NC_2(0,2k)$.
That is, $p_k$ is the rotation of $\pairpart^{\otimes k}$. We do this by induction. The $\mathring p_k$ for $k=2,3,4,5$ are among the generators, so we have the initial step and a couple of others already by assumption. We construct any $\mathring p_k$ by precomposing $\mathring p_{k-1}$ with $\PAdoubleimmerse\PAid\cdots\PAid$:
$$
\Partition{
\Pblock 0.5to0.8:1.7,2.3
\Pblock 0.5to0.8:2.7,3.8
\Pline (1.3,0.5) (1.3,1.1)
\Pline (1.3,1.1) (4.5,1.1)
\Pline (4.2,0.5) (4.2,0.8)
\Pline (4.2,0.8) (4.5,0.8)
\Pline (3.8,0.5) (3.8,-0.5)
\Pline (4.2,0.5) (4.2,-0.5)
\Pline (1.3,0.5) (0.8,-0.5)
\Pline (2.7,0.5) (3.2,-0.5)
\Pblock -0.5to-0.2:1.2,1.8
\Pblock -0.5to-0.2:2.2,2.8
\Pblock 0.5to0.2:1.7,2.3
\Ppoint0.5 \Pa:1.5,2.5,4
\Ppoint-0.5 \Pa:1,2,3,4
}
=\frac{1}{4}\left(
\Partition{
\Pblock 0.5to1.1:1.7,2.3
\Pblock 0.5to1.1:2.7,3.8
\Pline (1.3,0.5) (1.3,1.4)
\Pline (1.3,1.4) (4.5,1.4)
\Pline (4.2,0.5) (4.2,1.1)
\Pline (4.2,1.1) (4.5,1.1)
\Pline (3.8,0.5) (3.8,-0.5)
\Pline (4.2,0.5) (4.2,-0.5)
\Pline (1.3,0.5) (0.8,-0.5)
\Pline (2.7,0.5) (3.2,-0.5)
\Pblock -0.5to-0.2:1.2,1.8
\Pblock -0.5to-0.2:2.2,2.8
\Pblock 0.5to0.2:1.7,2.3
\Ppoint-0.5 \Pa:1,2,3,4
}
-
\Partition{
\Pblock 0.8to1.1:1.7,2.3
\Pblock 0.8to1.1:2.7,3.8
\Pline (1.3,0.8) (1.3,1.4)
\Pline (1.3,1.4) (4.5,1.4)
\Pline (4.2,0.8) (4.2,1.1)
\Pline (4.2,1.1) (4.5,1.1)
\Pline (3.8,0.8) (3.8,-0.5)
\Pline (4.2,0.8) (4.2,-0.5)
\Pline (1.3,0.5) (0.8,-0.5)
\Pline (2.7,0.5) (3.2,-0.5)
\Pline (1.3,0.8) (1.7,0.5)
\Pline (1.7,0.8) (1.3,0.5)
\Pline (2.3,0.8) (2.3,0.5)
\Pline (2.7,0.8) (2.7,0.5)
\Pblock -0.5to-0.2:1.2,1.8
\Pblock -0.5to-0.2:2.2,2.8
\Pblock 0.5to0.2:1.7,2.3
\Ppoint-0.5 \Pa:1,2,3,4
}
-
\Partition{
\Pblock 0.8to1.1:1.7,2.3
\Pblock 0.8to1.1:2.7,3.8
\Pline (1.3,0.8) (1.3,1.4)
\Pline (1.3,1.4) (4.5,1.4)
\Pline (4.2,0.8) (4.2,1.1)
\Pline (4.2,1.1) (4.5,1.1)
\Pline (3.8,0.8) (3.8,-0.5)
\Pline (4.2,0.8) (4.2,-0.5)
\Pline (1.3,0.5) (0.8,-0.5)
\Pline (2.7,0.5) (3.2,-0.5)
\Pline (1.3,0.8) (1.3,0.5)
\Pline (1.7,0.8) (1.7,0.5)
\Pline (2.3,0.8) (2.7,0.5)
\Pline (2.7,0.8) (2.3,0.5)
\Pline (1.3,0.5) (0.8,-0.5)
\Pline (2.7,0.5) (3.2,-0.5)
\Pblock -0.5to-0.2:1.2,1.8
\Pblock -0.5to-0.2:2.2,2.8
\Pblock 0.5to0.2:1.7,2.3
\Ppoint-0.5 \Pa:1,2,3,4
}
+
\Partition{
\Pblock 0.8to1.1:1.7,2.3
\Pblock 0.8to1.1:1.7,2.3
\Pblock 0.8to1.1:2.7,3.8
\Pline (1.3,0.8) (1.3,1.4)
\Pline (1.3,1.4) (4.5,1.4)
\Pline (4.2,0.8) (4.2,1.1)
\Pline (4.2,1.1) (4.5,1.1)
\Pline (3.8,0.8) (3.8,-0.5)
\Pline (4.2,0.8) (4.2,-0.5)
\Pline (1.3,0.5) (0.8,-0.5)
\Pline (2.7,0.5) (3.2,-0.5)
\Pline (1.3,0.8) (1.7,0.5)
\Pline (1.7,0.8) (1.3,0.5)
\Pline (2.3,0.8) (2.7,0.5)
\Pline (2.7,0.8) (2.3,0.5)
\Pline (1.3,0.5) (0.8,-0.5)
\Pline (2.7,0.5) (3.2,-0.5)
\Pblock -0.5to-0.2:1.2,1.8
\Pblock -0.5to-0.2:2.2,2.8
\Pblock 0.5to0.2:1.7,2.3
\Ppoint-0.5 \Pa:1,2,3,4
}
\right)
=\frac{n-2}{4}\mathring p_k+\frac{1}{4}\mathring p_3\otimes\mathring p_{k-2},
$$

As the second step, we prove that $\mathring p\in\Cat$ for every pairing $p\in\Part_2(0,2k)$. We use the element $\PAcross$, which allows to permute the two-points in $\mathring p$. We claim that any $p\in\Part_2(0,2k)$ can be obtained by such two-point permutations from some $\mathring p_{i_1}\otimes\cdots\otimes\mathring p_{i_l}$. Since we already proved that $\mathring p_i\in\Cat$ for every $i>1$ and $\PAcross\in\Cat$, this will finish the proof of the theorem. Note that thanks to the antisymmetrization, the order of the two points in a two-point is irrelevant (only affecting the $\pm$ sign).

So consider any $p\in\Part_2(0,2k)$. Take the first two-point and denote the corresponding points by $\rm pt_1$ and $\rm pt_2$. Take the point which is paired with $\rm pt_2$ and denote it by $\rm pt_3$. We denote by $\rm pt_4$ its neighbour that form a two-point with $\rm pt_3$. Perform a two-point permutation of $p$ such that $\rm (pt_3,pt_4)$ is the second two-point. Call $\rm pt_5$ the point, which is paired with $\rm pt_4$ and continue in this manner until we find some ${\rm pt}_{2i_1}$, which is paired with $\rm pt_1$. At this point, we have that $\mathring p$ is a two-point permutation of $\mathring p_{i_1}\otimes\mathring q$, where $q\in\Part_2(0,2k-2i_1)$. If we use mathematical induction, we may assume that $q$ is already a two-point permutation of $p_{i_2}\otimes\cdots\otimes p_{i_l}$. 
\end{proof}

\begin{lem}
\label{L2.projwedge}
Suppose $n\neq 4,6,8$. Then both $\PAconnecter$ and $\PAcross$ are generated by $\{\PApair,\PAimmerse\}$.
\end{lem}
\begin{proof}
Denote by $\Cat$ the category generated by $\{\PApair,\PAimmerse\}$. Recall that $\PAimmerse$ is a rotation of $\PAthree$. So, we can do the following computation in $\Cat$:
$$
\Partition{
\Pblock 0.5to0.8:1.2,1.8
\Pblock 0.5to0.8:2.2,3.3
\Pblock 0.5to1.1:0.8,3.7
\Pline (0.8,0.5) (0.8,-0.5)
\Pline (1.2,0.5) (1.2,-0.5)
\Pline (1.8,0.5) (1.8,-0.5)
\Pline (2.2,0.5) (2.2,-0.5)
\Pline (3.3,0.5) (2.8,-0.5)
\Pline (3.7,0.5) (4.2,-0.5)
\Pblock -0.5to-0.2:3.2,3.8
\Ppoint-0.5 \Pa:1,2,3,4
\Ppoint 0.5 \Pa:1,2,3.5
}
=\frac{1}{2}\left(
\Partition{
\Pblock 0.5to0.8:1.2,1.8
\Pblock 0.5to0.8:2.2,3.3
\Pblock 0.5to1.1:0.8,3.7
\Pline (0.8,0.5) (0.8,-0.5)
\Pline (1.2,0.5) (1.2,-0.5)
\Pline (1.8,0.5) (1.8,-0.5)
\Pline (2.2,0.5) (2.2,-0.5)
\Pline (3.3,0.1) (2.8,-0.5)
\Pline (3.7,0.1) (4.2,-0.5)
\Pline (3.3,0.5) (3.3, 0.1)
\Pline (3.7,0.5) (3.7, 0.1)
\Pblock -0.5to-0.2:3.2,3.8
\Ppoint-0.5 \Pa:1,2,3,4
}
-
\Partition{
\Pblock 0.5to0.8:1.2,1.8
\Pblock 0.5to0.8:2.2,3.3
\Pblock 0.5to1.1:0.8,3.7
\Pline (0.8,0.5) (0.8,-0.5)
\Pline (1.2,0.5) (1.2,-0.5)
\Pline (1.8,0.5) (1.8,-0.5)
\Pline (2.2,0.5) (2.2,-0.5)
\Pline (3.3,0.1) (2.8,-0.5)
\Pline (3.7,0.1) (4.2,-0.5)
\Pline (3.3,0.5) (3.7, 0.1)
\Pline (3.7,0.5) (3.3, 0.1)
\Pblock -0.5to-0.2:3.2,3.8
\Ppoint-0.5 \Pa:1,2,3,4
}
\right)=\frac{1}{2}(
\PAfour-\LPartition{}{0.3:1.2,1.8;0.3:3.2,3.8;0.6:2.2,4.2;0.9:0.8,2.8}[aaaa]
)$$

By rotation, this means that $\Cat$ contains the linear combinations $\PAconnecter-\PAinnercross$ and $\PAconnecter-\PAoutercross$. Now, we can do the following (we leave out the detailed computation now):
$$4(\PAconnecter-\PAinnercross)^2=(n-4)\PAconnecter-2\PAinnercross-2\PAoutercross+\PAcross+\PApairproj+\PAid\PAid.$$
Subtracting $2(\PAconnecter-\PAinnercross)$, $2(\PAconnecter-\PAoutercross)$, $\PApairproj$, $\PAid\PAid$ -- those all being elements of $\Cat$ -- we get that $\Cat$ must contain $(n-8)\PAconnecter+\PAcross$. Finally, squaring this element, we get
$$\left((n-8)\PAconnecter+\PAcross\right)^2=\frac{(n-8)^2(n-2)}{4}\PAconnecter+\frac{(n-8)^2}{4}\PApairproj+2(n-8)\PAoutercross+\PAid\PAid$$
Doing all the possible subtractions and multiplying by 4, we get that $\Cat$ contains $\alpha\PAconnecter$, where
$$\alpha=(n-8)^2(n-2)+8(n-8)=(n-4)(n-6)(n-8)$$
Consequently, $\PAconnecter\in\Cat$ unless $n=4,6,8$. Since we also proved that $(n-8)\PAconnecter+\PAcross\in\Cat$, we now have that $\PAcross\in\Cat$.
\end{proof}

\begin{lem}
\label{L3.projwedge}
Suppose $n\neq 4$. Then $\PAfive$ is generated by $\{\PApair,\PAthree,\PAfour,\PAcross\}$.
\end{lem}
\begin{proof}
First, we generate the following two elements of $\Cat$:
\begin{align}
2\Partition{
\Pblock 0.5to0.8:1.2,1.8
\Pblock 0.5to0.8:2.2,3.3
\Pblock 0.5to0.8:3.7,4.8
\Pblock 0.5to1.1:0.8,5.2
\Pline (0.8,0.5) (0.8,-0.5)
\Pline (1.2,0.5) (1.2,-0.5)
\Pline (1.8,0.5) (1.8,-0.5)
\Pline (2.2,0.5) (2.2,-0.5)
\Pline (3.3,0.5) (2.8,-0.5)
\Pline (3.7,0.5) (4.2,-0.5)
\Pline (4.8,0.5) (4.8,-0.5)
\Pline (5.2,0.5) (5.2,-0.5)
\Pblock -0.5to-0.2:3.2,3.8
\Ppoint-0.5 \Pa:1,2,3,4,5
\Ppoint 0.5 \Pa:1,2,3.5,5
}
=\PAfive-\LPartition{}{0.3:1.2,1.8;0.3:3.2,3.8;0.5:2.8,4.8;0.7:2.2,4.2;0.9:0.8,5.2}[aaaaa]\in\Cat,
\label{eq.PA1}\\
2\Partition{
\Pblock 0.5to0.8:1.2,2.3
\Pblock 0.5to0.8:2.7,3.8
\Pblock 0.5to0.8:4.2,4.8
\Pblock 0.5to1.1:0.8,5.2
\Pline (0.8,0.5) (0.8,-0.5)
\Pline (1.2,0.5) (1.2,-0.5)
\Pline (2.3,0.5) (1.8,-0.5)
\Pline (2.7,0.5) (3.2,-0.5)
\Pline (3.8,0.5) (3.8,-0.5)
\Pline (4.2,0.5) (4.2,-0.5)
\Pline (4.8,0.5) (4.8,-0.5)
\Pline (5.2,0.5) (5.2,-0.5)
\Pblock -0.5to-0.2:2.2,2.8
\Ppoint-0.5 \Pa:1,2,3,4,5
\Ppoint 0.5 \Pa:1,2.5,4,5
}
=\PAfive-\LPartition{}{0.3:2.2,2.8;0.3:4.2,4.8;0.5:1.8,3.8;0.7:1.2,3.2;0.9:0.8,5.2}[aaaaa]\in\Cat.
\label{eq.PA2}
\end{align}
Now, take the second element and permute the third and fourth two-point to obtain $\LPartition{}{0.3:1.2,1.8;0.3:3.2,3.8;0.5:2.8,4.8;0.7:2.2,4.2;0.9:0.8,5.2}[aaaaa]-\LPartition{}{0.5:1.8,4.2;0.3:3.2,4.8;0.3:2.2,2.8;0.7:1.2,3.8;0.9:0.8,5.2}[aaaaa]$. Adding the element~\eqref{eq.PA1}, we get
\begin{equation}
\label{eq.PA3}
\PAfive-\LPartition{}{0.5:1.8,4.2;0.3:3.2,4.8;0.3:2.2,2.8;0.7:1.2,3.8;0.9:0.8,5.2}[aaaaa]\in\Cat.
\end{equation}

Now we go another way: Start with the element~\eqref{eq.PA1} and precompose it with $\PAid\PAconnecter\PAid\PAid$. We obtain $\frac{1}{4}(n-3)\PAfive-\frac{1}{4}\LPartition{}{0.5:1.8,4.2;0.3:3.2,4.8;0.3:2.2,2.8;0.7:1.2,3.8;0.9:0.8,5.2}[aaaaa]$. Multiplying by four and adding \eqref{eq.PA3}, we finally get $(n-4)\PAfive\in\Cat$.
\end{proof}

\begin{proof}[Proof of Theorem~\ref{T.projwedge}]
Follows directly from the lemmata above. Lemma~\ref{L2.projwedge} tells us that $\PApair$ and $\PAimmerse$ generate $\PAfour$ and $\PAcross$. Lemma~\ref{L3.projwedge} shows that those together generate $\PAfive$. Finally, Lemma~\ref{L1.projwedge} shows that all those together already generate the whole category $\Pair_n^\circ$.
\end{proof}

In the formulation of Theorem~\ref{T.projwedge}, we needed to make the assumption $n\neq 2,4,6,8$. We can easily repair the formulation to include the cases $n=6,8$ as well. Notice that the only place, where we needed the assumption $n\neq 6,8$ was Lemma~\ref{L2.projwedge}. So, we can modify the formulation of the theorem as follows:

\begin{prop}
Consider $n\neq 2,4$. Then the category $\Pair_n^\circ$ is generated by the set $\{\PApair,\PAthree,\PAfour,\PAcross\}$.
\end{prop}

We will actually need a slightly more technical result in the sequel.
\begin{prop}
Consider $n\neq 2,4$. Then the category $\Pair_n^\circ$ is generated by the set $\{\PApair,\PAimmerse,p\}$, where
\begin{equation}
\label{eq.folded}
p=
\PAfour
-\LPartition{}{0.3:1.2,1.8;0.3:3.2,3.8;0.6:0.8,2.8;0.9:2.2,4.2}[aaaa]
-\LPartition{}{0.3:2.2,2.8;0.5:1.8,3.8;0.7:1.2,3.2;0.9:0.8,4.2}[aaaa]
+\LPartition{}{0.3:1.2,1.8;0.3:3.2,3.8;0.6:0.8,2.2;0.6:2.8,4.2}[aaaa]
+\LPartition{}{0.3:2.2,2.8;0.5:1.8,3.2;0.7:1.2,3.8;0.9:0.8,4.2}[aaaa]
+\LPartition{}{0.3:1.2,2.8;0.5:0.8,3.2;0.7:2.2,3.8;0.9:1.8,4.2}[aaaa].
\end{equation}
\end{prop}
\begin{proof}
We only need to show that the given intertwiners generate $\PAfour$ and $\PAcross$. We do this by modifying the proof of Lemma~\ref{L2.projwedge}.

In Lemma~\ref{L2.projwedge}, we showed that $\{\PApair,\PAimmerse\}$ actually generate the following elements (for which we did not yet need the assumption $n\neq 6,8$): The element $\PAfour-\LPartition{}{0.3:1.2,1.8;0.3:3.2,3.8;0.9:2.2,4.2;0.6:0.8,2.8}[aaaa]$. A rotation of this one is the element $\PAfour-\LPartition{}{0.3:2.2,2.8;0.5:1.8,3.8;0.7:1.2,3.2;0.9:0.8,4.2}[aaaa]$. We also construct the element $(n-8)\PAconnecter+\PAcross$, which can be rotated into $(n-8)\PAfour+\LPartition{}{0.3:1.2,2.8;0.5:0.8,3.2;0.7:2.2,3.8;0.9:1.8,4.2}[aaaa]$.

Subtracting those together with $\LPartition{}{0.3:1.2,1.8;0.3:3.2,3.8;0.6:0.8,2.2;0.6:2.8,4.2}[aaaa]$, $\LPartition{}{0.3:2.2,2.8;0.5:1.8,3.2;0.7:1.2,3.8;0.9:0.8,4.2}[aaaa]\in\Cat$ from the element \eqref{eq.folded}, we get that $\Cat$ contains $(n-7)\PAfour$. Consequently, it must contain both $\PAfour$ and $\PAcross$. (For $n=7$, we can use Lemma~\ref{L2.projwedge} directly.)
\end{proof}

Finally, let us reveal the meaning of the mysterious linear combination $p$ from Equation~\eqref{eq.folded}. One can easily check that given a tuple of indices $i_1$, $j_1$, $i_2$, $j_2$, $i_3$, $j_3$, $i_4$, $j_4$ such that $i_k\neq j_k$, we have
\begin{equation}\label{eq.folded2}
[\breve T_p^{(n)}]^{(i_1,j_1),(i_2,j_2),(i_3,j_3),(i_4,j_4)}=\begin{cases}1&\text{if the indices can be paired}\\0&\text{otherwise}\end{cases}
\end{equation}
Indeed, the individual partitions just depict the ways of how the indices can be paired. (The minus signs are there because of the crossings to compensate the minus sign given by the deformed functor $\breve T^{(n)}$.)

Thus, we have the following Corollary
\begin{cor}
\label{C2.antiprojwedge}
Consider $n\neq 2,4$. Then the representation category $\RCat_{\hat PO_n^{-1}}$ is generated by $\{\breve T^{(n)}_{\PApair},\breve T^{(n)}_{\PAimmerse},\breve T_p^{(n)}\}$, where $\breve T_p^{(n)}$ is given by Eq.~\eqref{eq.folded2}.
\end{cor}

\subsection{Proof of Theorem~\ref{T.fcube}}
\label{secc.fcubeproof}
Denote by $v$ the fundamental representation of $O_{n+1}^{-1}$. Theorem~\ref{T.fcube} says that $\Aut^+FQ_n$ is isomorphic with $PO_{n+1}^{-1}$ acting through $u$ with $\hat u=\bigoplus v^{\brewedge 2i}$. First step of the proof was provided by Corollary~\ref{P.antiprojwedge}, where we showed that $v\brewedge v$ is a faithful representation of $PO_{n+1}^{-1}$ and hence the whole $\bigoplus v^{\brewedge 2i}$ must be a faithful representation.

Secondly, we show that $PO_{n+1}^{-1}$ acts on $FQ_{n+1}$, which then implies that $PO_{n+1}^{-1}\subset\Aut^+ FQ_{n+1}$. But there is no work to be done here as this is just a restriction of the action of $O_{n+1}^{-1}$ on $Q_{n+1}$. Indeed, recall that we have the coaction $\alpha\colon C(Q_{n+1})\to C(Q_{n+1})\otimes \Olg(O_{n+1}^{-1})$ by $\tau_j\mapsto\sum_{i=1}^n \tau_i v^i_j$ and that $C(FQ_{n+1})$ is the subalgebra of $C(Q_{n+1})$ generated by even polynomials in $\tau_j$. Restricting to this subalgebra, we get exactly the desired coaction $C(FQ_{n+1})\to C(FQ_{n+1})\otimes \Olg(PO_{n+1}^{-1})$.

Finally, we need to prove the converse inclusion $\Aut^+ FQ_{n+1}\subset PO_{n+1}^{-1}$. So, we need to show that $\hat u^{(2)}$ is a representation of $PO_{n+1}^{-1}$. Assume for a moment that $n+1\neq 2,4,6,8$. Thanks to Corollary~\ref{C.antiprojwedge}, we only need to show that $\breve T^{(n+1)}_{\PApair}$ and $\breve T^{(n+1)}_{\PAimmerse}$ are intertwiners of $\hat u^{(1)}$. The first one is just the orthogonality, which is automatic as Fourier transform preserves orthogonality. The second one is easy to obtain by looking at the Fourier transform of $T^{(N)}_{\forkpart}\in\Mor(u,u^{\otimes 2})$ projected to $\tilde V_2$. Its entries are then
$$[T^{(N)}_{\forkpart}]^{(i_1,j_1),(i_2,j_2)}_{(i_3,j_3)}=
\begin{cases}
1&\text{if $(i_1,j_1,i_2,j_2,i_3,j_3)$ can be paired,}\\
0&\text{otherwise}.
\end{cases}$$
Note that due to the antisymetrization, we need to also assume $i_1\neq j_1$, $i_2\neq j_2$, and $i_3\neq j_3$. We claim that this exactly matches the intertwiner $\breve T^{(n+1)}_{\PAimmerse}$. Indeed, the latter is just a symmetrization of $\breve T^{(n+1)}_{\immersepart}=T^{(n+1)}_{\immersepart}$. It is clear that the only way how to pair a tuple of indices $i_1\neq j_1$, $i_2\neq j_2$, $i_3\neq j_3$ is according to the partition $\immersepart$ (up to symmetrization, i.e. permuting neighbouring pairs of points). This finishes the proof for the case $n\neq 6,8$.

For the cases $n=6,8$, consider the intertwiner of $\hat u^{(1)}$ induced by $\fourpart$. Again, the entries of $T_{\fournum}$ are zeros and ones depending on whether the indices can be paired or not. That is, $T_{\fournum}=N\breve T^{(n+1)}_p$, where $\breve T^{(n+1)}_p$ is given by Equation~\eqref{eq.folded2}. Thus, the result follows from Corollary~\ref{C2.antiprojwedge}.\qed

\subsection{Open problem}

Let us again finish with some open problem. The hyperoctahedral group $H_n$ can be seen not only as the symmetry group of the hypercube $Q_n$, but also as the symmetry group of $n$ copies of a segment $K_2\sqcup\cdots\sqcup K_2$. While the quantum symmetry of the former is $O_n^{-1}$, the quantum symmetry of the latter is $H_n^+$, which are two distinct quantum groups. We have just proven that $PO_n^{-1}$ is the quantum symmetry of the folded hypercube $FQ_n$. This result suggests the following question:

\begin{quest}
Is there some graph, whose quantum symmetry is described by the quantum group $PH_n^+$ for some $n$? Does $PH_n^+$ act on the set of $N$ points for some $N$ at all? (That is, do we have $PH_n^+\subset S_N^+$ for some $N$?)
\end{quest}

This is related to a big question on whether there is a quantum analogue of the Frucht theorem, which was discussed recently in~\cite{BMC21}.

\section{Hamming graphs}

\emph{Hamming graph} $H(n,m)$ is the Cayley graph of the group $\Gamma=\Z_m^n$ with respect to the generating set $S=\{a\epsilon_i\mid i=1,\dots,n,\;a=1,\dots,m-1\}$, where $\epsilon_i=(0,\dots,0,1,0,\dots,0)$ is the generator of the $i$-th copy of $\Z_m$. In other words, vertices of $H(n,m)$ are $n$-tuples of numbers $a=0,\dots,m-1$ (that is, indeed, $V=\Z_m^n$) and two such tuples are connected with an edge if and only if they differ in exactly one coordinate.

Another possible description is using the Cartesian product of graphs (see Section~\ref{secc.Cartesian}): Hamming graph $H(n,m)$ is the $n$-fold Cartesian product of the full graph $K_m$, that is, $H(n,m)=K_m\sqprod\cdots\sqprod K_m$.

The classical automorphism group of $H(n,m)$ is known to be the wreath product $S_m\wr S_n$. About the quantum automorphism group, only partial results are known so far:
\begin{itemize}
\item $n=1$: $H(1,m)$ is the full graph $K_m$, so $\Aut^+H(1,m)=S_m^+$.
\item $m=1$: $H(n,1)$ contains just a single vertex.
\item $m=2$: $H(n,2)$ is the hypercube $Q_n$, so $\Aut^+H(n,2)=O_n^{-1}$.
\item $m=3$: $H(n,3)$ was proven to have no quantum symmetries \cite{Sch20dt}, so $\Aut^+H(n,3)=\Aut H(n,3)=S_3\wr S_n$.
\item $m>3$: $H(n,m)$ was proven to have some quantum symmetries \cite{Sch20dt}, but the explicit quantum group was not known.
\end{itemize}

We are going to answer the question about the quantum automorphism group of Hamming graphs in full generality in the following theorem (by which we also answer Question~8.2(i) of Simon Schmidt's PhD thesis~\cite{SchThesis}):

\begin{thm}
\label{T.Hamming}
Consider $m\in\N\setminus\{1,2\}$. Then $\Aut^+H(n,m)\simeq S_m^+\wr S_n$.
\end{thm}

Before formulating the proof itself, we would like to explain some important ingredients more in detail.

\subsection{Full graph}

As we just mentioned, a special case for $n=1$ is the full graph $K_m$. Of course, we know that the quantum symmetry group of the full graph is the free symmetric quantum group $S_m^+$. Nevertheless, we would like to use this simple example to point out a certain subtlety that one needs to keep in mind when working with cyclic groups $\Z_m$ for $m\neq 2$.

So, the full graph $K_m$ is the Cayley graph corresponding to the group $\Z_m$ and the generating set consisting of all elements of the group except for identity, so $K_m=\Cay(\Z_m,\Z_m\setminus\{0\})$. We denote simply by $a=0,\dots,m-1$ the elements of $\Z_m$ and by $\tau_a\in\Irr\Z_m$ the characters $\tau_a(b)=\gamma^{ab}$, where $\gamma$ is some primitive $m$-th root of unity. The spectrum of $K_m$ is hence computed as
$$\lambda_a=\sum_{b=1}^{m-1}\tau_a(b)=\sum_{b=1}^{m-1}\gamma^{ab}=
\begin{cases}
m-1&a=0,\\
-1&\text{otherwise.}
\end{cases}$$

Now comes the important point we wanted to make in this subsection: The Fourier transform $\F$ on $\Z_m$, that is, the transformation of the bases $(\delta_a) \to (\tau_a)$ is unitary, but not orthogonal! Its entries are $[\F]^a_b=\gamma^{ab}$, so they are obviously not real (unless $m=2$). To be more concrete, the basis elements $\tau_a$, which are the columns of $\F$ are not self-adjoint, but satisfy $\tau_a^*=\tau_{m-a}$.

Consequently, if we denote $\hat S_m:=\F^{-1}S_m\F$ and $\hat S_m^+:=\F^{-1}S_m^+\F$ the symmetric group and the free symmetric quantum group represented by the Fourier transform of the standard permutation matrices, then those matrix (quantum) groups are not orthogonal. Instead, they satisfy $\hat S_m\subset\hat S_m^+\subset O^+(F)\subset U_m^+$ with $F\in M_m(\C)$ defined by $F^a_b=\delta_{a+b,m}$ (indices modulo $m$). That is, $u^{a\,*}_b=u^{m-a}_{m-b}$. (Here, $U_n^+$ denotes the free unitary quantum group \cite{Wan95free}.)

Therefore, if we study the intertwiners of $\hat S_m\simeq\Aut^+ K_m$ in this basis, then instead of the familiar maps such as $T^{(m)}_{\pairpart}$, $T^{(m)}_{\mergepart}$, $T^{(m)}_{\connecterpart}$, we discover their Fourier transforms, which may look rather exotic.

\begin{obs}
The category $\RCat_{\hat S_m}$ is generated by $\hat T^{(m)}_{\pairpart}$ and $\hat T^{(m)}_{\mergepart}$, where
$$[\hat T^{(m)}_{\pairpart}]^{ab}=\delta_{a+b,m},\qquad [\hat T^{(m)}_{\mergepart}]_{ab}^c=\delta_{a+b,c}.$$
\end{obs}

\subsection{Wreath product of quantum groups}

We should also explain, what does the $\wr$ sign in the formulation of Theorem~\ref{T.Hamming} means. It is not the free wreath product of quantum groups, but the \emph{classical} one. Before specifying, what we mean by a classical wreath product of quantum groups, let us recall the free definition by Bichon \cite{Bic04}.

\begin{defn}
Let $G\subset U_m^+$ be a quantum group with fundamental representation $v=(v^a_b)$ and let $H\subset S_n^+$ be a quantum permutation group with fundamental representation $w=(w^i_j)$. Then we define the \emph{free wreath product} of $G$ and $H$ to be the quantum group
$$G\wr_* H:=(\Olg(G)^{* n}* \Olg(H),u),\quad\text{where}\quad u^{ai}_{bj}=v^{ai}_bw^i_j,$$
where we denote by $v^i=(v^{ai}_b)_{a,b}$ the fundamental representations of the $n$ copies of $G$ occurring in the definition of $G\wr_* H$.
\end{defn}

\begin{rem}
\label{R.fwr}
Let us state a few important remarks to this definition.
\begin{enumerate}
\renewcommand{\theenumi}{\alph{enumi}}
\item Although we defined the free wreath product $G\wr_* H$ for compact matrix quantum groups using their fundamental representations, the original definition of Bichon is formulated for arbitrary compact quantum group $G$ not depending on its particular fundamental representation (see~\cite{Bic04} for details). In particular, if we take some other matrix realization $G'\simeq G$ of the quantum group $G$, then $G'\wr_* H$ is isomorphic to $G\wr_* H$.
\item As in the classical case, the free wreath product $G\wr_* H$ has a sort of a (free) semidirect product structure. What we mean by this is the following:
\item The matrix $w$ is a representation of $G\wr_* H$. Therefore, $H$ can be seen as a quotient of $G\wr_* H$.
\item On the other hand, the matrices $v^i$ are not representations of $G\wr_* H$ -- the coproduct is mixing (quantum-permuting) the indices $i$ non-trivially:
$$\Delta(v^{ai}_b)=\sum_{c,k}v^{ai}_cw^i_k\otimes v^{ck}_b.$$
\item We can express
$$w^i_j=\sum_b u^{ai*}_{bj}u^{ai}_{bj}=\sum_a u^{ai*}_{bj}u^{ai}_{bj},\qquad v^{ai}_b=\sum_j u^{ai}_{bj}$$
That is, the entries $u^{ai}_{bj}$ indeed generate the whole algebra $\Olg(G)^{\otimes n}\otimes \Olg(H)$. This remark is essential to notice that the definition above is a good definition of $G\wr H$ as a compact matrix quantum group.
\end{enumerate}
\end{rem}

Now the classical wreath product is supposed to be given by passing from the free product to the tensor product. So, define $\Olg(G\wr H):=\Olg(G)^{\otimes n}\otimes\Olg(H)$. 

\begin{lem}
Consider a quantum group $G\subset U_m^+$ and a classical permutation group $H\subset S_n$. Then the comultiplication $\Delta\colon\Olg(G\wr_*H)\to\Olg(G\wr_*H)\otimes\Olg(G\wr_*H)$ passes to the quotient $\Olg(G\wr H)$
\end{lem}
\begin{proof}
Denote $\Delta':=(q\otimes q)\circ\Delta$, where $q$ is the projection $\Olg(G\wr_*H)\to\Olg(G\wr H)$. We only need to prove that $\Delta'(v^{ai}_b)\Delta'(v^{cj}_d)=\Delta'(v^{cj}_d)\Delta'(v^{ai}_b)$ whenever $i\neq j$ and $\Delta'(v^{ai}_b)\Delta'(w^k_l)=\Delta'(w^k_l)\Delta'(v^{ai}_b)$. Both is quite straightforward. Let's have a look on the first one:
$$\Delta'(v^{ai}_b)\Delta'(v^{cj}_d)=\sum_{x,k,y,l}v^{ai}_xw^i_kv^{cj}_yw^j_l\otimes v^{xk}_bv^{yl}_d$$
Now, notice that the factors in the left $\otimes$-factor can be arbitrarily permuted. Assuming $k\neq l$, the same holds for the right $\otimes$-factor. For $k=l$, we have $w^i_kw^j_l=0$, so the left $\otimes$-factor equals to zero. Consequently, we see that the coproducts indeed commute as we needed. The second condition is proven the same way using the fact that $\Delta(w^i_j)=\sum_k w^i_k\otimes w^k_j$.
\end{proof}

\begin{defn}
For a quantum group $G\subset U_m^+$ and a classical permutation group $H\subset S_n$, we define their wreath product $G\wr H$ to be the quantum subgroup of $G\wr_*H$ corresponding to the quotient algebra $\Olg(G\wr H)=\Olg(G)^{\otimes n}\otimes\Olg(H)$.
\end{defn}

\subsection{Cartesian product of graphs}
\label{secc.Cartesian}

Given two graphs $X_1$ and $X_2$, we define their \emph{Cartesian product} to be the graph $X_1\sqprod X_2$ with the vertex set $V(X_1\sqprod X_2)=V(X_1)\times V(X_2)$ and with an edge $((v_1,v_2),(w_1,w_2))\in E(X_1\sqprod X_2)$ if and only if $(v_1,w_1)\in E(X_1)$ and $v_2=w_2$ or if $(v_2,w_2)\in E(X_2)$ and $v_1=w_1$. Alternatively, we can describe the Cartesian product by its adjacency matrix $A_{X_1\sqprod X_2}=A_{X_1}\otimes\nobreak I_{X_2}+I_{X_1}\otimes\nobreak A_{X_2}$, where $I_{X_i}$ denotes the identity matrix.

The Cartesian product of graphs is associative and we can conveniently describe the product of $n$ given graphs by the adjacency matrix
$$A_{X_1\sqprod\cdots\sqprod X_n}=A_{X_1}\otimes I_{X_2}\otimes I_{X_3}\otimes\cdots\otimes I_{X_n}+I_{X_1}\otimes A_{X_2}\otimes I_{X_3}\otimes\cdots\otimes I_{X_n}+\cdots$$

It is well known that if $G$ acts on a finite space or a graph $X$ by $\delta_a\mapsto\sum_b\delta_b\otimes v^b_a$, then $G\wr S_n$ acts on the $n$-fold disjoint union $X\sqcup\cdots\sqcup X$ by $\delta_{ai}\mapsto\sum_{b,j}\delta_{bj}\otimes u^{bj}_{ai}$. (Notice that $C(X\sqcup\cdots\sqcup X)=C(X)\oplus\cdots\oplus C(X)$; the indices $i,j$ are indexing the $n$ copies of $X$ or $C(X)$ here.)

Now, consider the Cartesian product $X\sqprod\cdots\sqprod X$. In this case, we have $C(X\sqprod\cdots\sqprod X)=C(X)\otimes\cdots\otimes C(X)$. Consider a basis $(x_i)_{i=0}^{m-1}$ of $C(X)$ such that $x_0=1_{C(X)}$ (if $X$ is a regular graph, then we can consider the basis of eigenvectors of $A_X$). Denote by $\hat v^a_b$ the entries of the action of $G$ on $X$ in this basis, so $x_a\mapsto\sum_bx_b\otimes\hat v^b_a$. Denote $x_{ai}:=1_{C(X)}\otimes\cdots\otimes 1_{C(X)}\otimes x_a\otimes 1_{C(X)}\otimes\cdots\otimes 1_{C(X)}$, where the $x_a$ is on the $i$-th place. In the following we are going to prove that $x_{ai}\mapsto\sum_{b,j}x_{bj}\otimes\hat u^{bj}_{ai}$ extends to an action of $G\wr S_n$ on $X\sqprod\cdots\sqprod X$, where $\hat u^{bj}_{ai}=\hat v^{bj}_aw^j_i$.

First, assume for a moment that this action really exists. Then it is easy to determine, how it must act on the basis $x_{a_1,\dots,a_n}:=x_{a_1}\otimes\cdots\otimes x_{a_n}$ of $C(X\sqprod\cdots\sqprod X)$:
\begin{align*}
x_{a_1,\dots,a_n}
&\mapsto \sum_{b_1,\dots,b_n=0}^{m-1}\sum_{k_1,\dots,k_n=1}^n x_{b_1k_1}\cdots x_{b_nk_n}\otimes \hat u^{b_1k_1}_{a_11}\cdots \hat u^{b_nk_n}_{a_nn}\\
&=\sum_{b_1,\dots,b_n=0}^{m-1}\sum_{\sigma\in S_n} x_{b_1\sigma(1)}\cdots x_{b_n\sigma(n)}\otimes \hat u^{b_1\sigma(1)}_{a_11}\cdots \hat u^{b_n\sigma(n)}_{a_nn}\\
&=\sum_{b_1,\dots,b_n=0}^{m-1}\sum_{\sigma\in S_n} x_{b_{\sigma^{-1}(1)}1}\cdots x_{b_{\sigma^{-1}(n)}n}\otimes \hat u^{b_1\sigma(1)}_{a_11}\cdots \hat u^{b_n\sigma(n)}_{a_nn}\\
&=\sum_{b_1,\dots,b_n=0}^{m-1}\sum_{\sigma\in S_n} x_{b_11}\cdots x_{b_nn}\otimes \hat u^{b_{\sigma(1)}\sigma(1)}_{a_11}\cdots \hat u^{b_{\sigma(n)}\sigma(n)}_{a_nn}\\
&=\sum_{b_1,\dots,b_n=0}^{m-1}x_{b_1,\dots,b_n}\otimes\hat{\tilde u}^{b_1,\dots,b_n}_{a_1,\dots,a_n},
\end{align*}
where $\hat{\tilde u}^{b_1,\dots,b_n}_{a_1,\dots,a_n}=\sum_{\sigma\in S_n}\hat u^{b_{\sigma(1)}\sigma(1)}_{a_11}\cdots\hat u^{b_{\sigma(n)}\sigma(n)}_{a_nn}$. We can also change the basis to the standard one and obtain $\delta_{a_1,\dots,a_n}\mapsto\sum_{b_1,\dots,b_n}\delta_{b_1,\dots,b_n}\otimes\tilde u^{b_1,\dots,b_n}_{a_1,\dots,a_n}$, where $\tilde u$ is given by a formula analogous to the one for $\hat{\tilde u}$.

\begin{lem}
\label{L.wr}
Let $G$ be a compact matrix quantum group with fundamental representation $v$. Then
$$\tilde u^{b_1,\dots,b_n}_{a_1,\dots,a_n}
:=\sum_{\sigma\in S_n}u^{b_{\sigma(1)}\sigma(1)}_{a_11}\cdots u^{b_{\sigma(n)}\sigma(n)}_{a_nn}
=\sum_{k_1,\dots,k_n=1}^nu^{b_{k_1}k_1}_{a_11}\cdots u^{b_{k_n}k_n}_{a_nn}$$
is a representation of $G\wr S_n$. If $G\subset S_m^+$, then $\tilde u$ is faithful.
\end{lem}
\begin{proof}
First, we should prove the second equality in the formula. To see this, it is enough to notice that the product $u^{b_{k_1}k_1}_{a_11}\cdots u^{b_{k_n}k_n}_{a_nn}$ equals to zero whenever $k_i=k_j$ for some $i\neq j$ (since $w^{k_i}_iw^{k_j}_j=0$).

Proving that $\tilde u$ is a representation of $G\wr S_n$ is a straightforward computation:
\begin{align*}
\Delta(u^{b_1,\dots,b_n}_{a_1,\dots,a_n})
&=\sum_{\sigma\in S_n}\Delta(u^{b_{\sigma(1)}\sigma(1)}_{a_11})\cdots\Delta(u^{b_{\sigma(n)}n}_{a_nn})\\
&=\sum_{\sigma\in S_n}\sum_{\substack{c_1,\dots,c_n\\k_1,\dots,k_n}}u^{b_{\sigma(1)}\sigma(1)}_{c_1k_1}\cdots u^{b_{\sigma(n)}\sigma(n)}_{c_nk_n}\otimes u^{c_1k_1}_{a_11}\cdots u^{c_nk_n}_{a_nn}\\
&=\sum_{d_1,\dots,d_n}\sum_{\pi,\rho\in S_n}u^{b_{\rho(1)}\rho(1)}_{d_11}\cdots u^{b_{\rho(n)}\rho(n)}_{d_nk_n}\otimes u^{d_1\pi(1)}_{a_11}\cdots u^{d_n\pi(n)}_{a_nn}\\
&=\sum_{d_1,\dots,d_n}u^{b_1,\dots,b_n}_{d_1,\dots,d_n}\otimes u^{d_1,\dots,d_n}_{a_1,\dots,a_n}
\end{align*}
To get from the second to the third line, we need to notice several things: First, as we already mentioned, the product $u^{c_1k_1}_{a_11}\cdots u^{c_nk_n}_{a_nn}$ equals to zero unless $k_1,\dots,k_n$ is a permutation of $1,\dots,n$. Hence, we can denote this permutation by $\pi$, so $k_i=\pi(i)$. Secondly, the terms of the product mutually commute, so we can reorder the first product as $u^{b_{\sigma(1)}\sigma(1)}_{c_1k_1}\cdots u^{b_{\sigma(n)}\sigma(n)}_{c_nk_n}=u^{b_{\sigma(\pi^{-1}(1))}\sigma(\pi^{-1}(1))}_{c_{\pi^{-1}(1)}1}\cdots u^{b_{\sigma(\pi^{-1}(n))}\sigma(\pi^{-1}(n))}_{c_{\pi^{-1}(n)}n}$. Finally, we denote $\rho:=\sigma\circ\pi^{-1}$ and $d_i:=c_{\pi^{-1}(i)}$.

Assume now that $G\subset S_m$ for some $m$. The proof of the last statement -- that $\tilde u$ is faithful -- gets a bit easier if we work in the basis $(x_a)_{a=0}^{m-1}$ of $C(X)$ such that $x_0=1_{C(X)}$ since $x_0$ is an invariant vector of $v$. So denote by $\hat v^a_b$ the entries of $v$ in the basis $(x_a)$ and similarly $\hat u^{ai}_{bj}:=\hat v^{ai}_bw^i_j$. We have then $\hat v^0_0=1$ and $\hat v^0_b=0=\hat v^a_0$ for every $a$, $b$. We need to show that the entries of $\hat{\tilde u}$ already generate the whole algebra $\Olg(G\wr S_n)$. Of course, it is enough to show that it generates the generators $\hat v^{ai}_b$ and $w^i_j$. We claim that $\hat{\tilde u}^{0,\dots,0,a,0,\dots,0}_{0,\dots,0,b,0,\dots,0}=\hat u^{ai}_{bj}$, where the $a$ is on the $i$-th position and the $b$ is on the $j$-th position on the left-hand side. Indeed, we get
\begin{equation}\label{eq.hattilu}
\hat{\tilde u}^{0,\dots,0,a,0,\dots,0}_{0,\dots,0,b,0,\dots,0}=\sum_{\substack{k_1,\dots,k_n=1\\\text{except for $k_j:=i$}}}^n \hat v^{ai}_b\,w^{k_1}_1\cdots w^{k_n}_n=\hat v^{ai}_bw^i_j=\hat u^{ai}_{bj}.\qedhere
\end{equation}
\end{proof}

\begin{prop}
\label{P.wr}
Let $\Gamma$ be a graph. Then $\Aut^+(\Gamma\sqprod\cdots\sqprod\Gamma)\supset(\Aut^+\Gamma)\wr S_n$. More precisely, $G\wr S_n$ acts faithfully on the $n$-fold product $\Gamma\sqprod\cdots\sqprod\Gamma$ by $\delta_{a_1,\dots,a_n}\mapsto\sum_{b_1,\dots,b_n}\delta_{b_1,\dots,b_n}\otimes\tilde u^{b_1,\dots,b_n}_{a_1,\dots,a_n}$.

\end{prop}
\begin{proof}
Notice that we can express $\tilde u$ in a more ``matricial way''
$$\tilde u=\sum_{\sigma\in S_n}T_\sigma(v^{\sigma(1)}\otimes\cdots\otimes v^{\sigma(n)})\delta_\sigma,$$
where $T_\sigma\colon(\C^m)^{\otimes n}\to(\C^m)^{\otimes n}$ is the linear operator permuting the tensor factors according to $\sigma$ and $\delta_{\sigma}:=w_1^{\sigma(1)}\cdots w_n^{\sigma(n)}$ (it is actually indeed the delta function $S_n\to\C$ mapping $\delta_\sigma(\pi)=\delta_{\sigma\pi}$). We will use this matrix approach throughout the proof, but one could of course also rewrite the computation in terms of the matrix entries.

We want to show that $(\Aut^+\Gamma)\wr S_n$ represented by the faithful representation $\tilde u$ is a quantum subgroup of $\Aut^+(\Gamma\sqprod\cdots\sqprod\Gamma)$. To do this, we need to show that the representation category associated to $\tilde u$ contains the generators of the category $\RCat_{\Aut^+(\Gamma\sqprod\cdots\sqprod\Gamma)}$, which are $T_{\singleton}^{(m^n)}$, $T_{\mergepart}^{(m^n)}$, and $\tilde A$, where $\tilde A$ is the adjacency matrix of $\Gamma\sqprod\cdots\sqprod\Gamma$.

We start with the singleton $T^{(m^n)}_{\singleton}$, which is the easiest one. Notice first that $T_{\singleton}^{(m^n)}=T_{\singleton}^{(m)}\otimes\cdots\otimes T_{\singleton}^{(m)}$. Consequently,
\begin{align*}
\tilde uT_{\singleton}^{(m^n)}
&=\sum_{\sigma\in S_n}T_\sigma(v^{\sigma(1)}\otimes\cdots\otimes v^{\sigma(n)})T_{\singleton}^{(m^n)}\delta_\sigma\\
&=\sum_{\sigma\in S_n}T_\sigma(v^{\sigma(1)}T_{\singleton}^{(m)}\otimes\cdots\otimes v^{\sigma(n)}T_{\singleton}^{(m)})\delta_\sigma\\
&=\sum_{\sigma\in S_n}T_\sigma(T_{\singleton}^{(m)}\otimes\cdots\otimes T_{\singleton}^{(m)})\delta_\sigma
 =T_{\singleton}^{(m^n)},
\end{align*}
where we used the fact that $\sum_{\sigma\in S_n}\delta_\sigma=1_{C(S_n)}$.

To show the second intertwiner relation, denote by $R$ the natural ``disentangling operator'' $(\C^m\otimes\C^m)^{\otimes n}\to (\C^m)^{\otimes n}\otimes(\C^m)^{\otimes n}$ mapping $(x_1\otimes y_1)\otimes\cdots\otimes(x_n\otimes y_n)\mapsto (x_1\otimes\cdots\otimes x_n)\otimes(y_1\otimes\cdots\otimes y_n)$. Then we can write $T_{\mergepart}^{(m^n)}=(T_{\mergepart}^{(m)}\otimes\cdots\otimes T_{\mergepart}^{(m)})R$. So,
\begin{align*}
\tilde uT_{\mergepart}^{(m^n)}
&=\sum_{\sigma\in S_n}T_\sigma(v^{\sigma(1)}\otimes\cdots\otimes v^{\sigma(n)})T_{\mergepart}^{(m^n)}\delta_\sigma\\
&=\sum_{\sigma\in S_n}T_\sigma(v^{\sigma(1)}T_{\mergepart}^{(m)}\otimes\cdots\otimes v^{\sigma(n)}T_{\mergepart}^{(m)})R\delta_\sigma\\
&=\sum_{\sigma\in S_n}T_\sigma(T_{\mergepart}^{(m)}(v^{\sigma(1)}\otimes v^{\sigma(1)})\otimes\cdots\otimes T_{\mergepart}^{(m)}(v^{\sigma(n)}\otimes v^{\sigma(n)}))R\delta_\sigma\\
&=\sum_{\sigma\in S_n}T_\sigma T_{\mergepart}^{(m^n)}((v^{\sigma(1)}\otimes\cdots\otimes v^{\sigma(n)})\otimes(v^{\sigma(1)}\otimes\cdots\otimes v^{\sigma(n)}))\delta_\sigma\\
&=\sum_{\sigma\in S_n}T_{\mergepart}^{(m^n)}(T_\sigma\otimes T_\sigma)((v^{\sigma(1)}\otimes\cdots\otimes v^{\sigma(n)})\otimes(v^{\sigma(1)}\otimes\cdots\otimes v^{\sigma(n)}))\delta_\sigma\\
&=T_{\mergepart}^{(m^n)}(\tilde u\otimes\tilde u).
\end{align*}

Finally, we prove that $\tilde u$ commutes with $\tilde A:=\sum_{i=1}^n \id\otimes\cdots\otimes\id\otimes A\otimes\id\otimes\cdots\otimes\id$, where $A$ is the adjacency matrix of $\Gamma$ and in each summand it appears at the $i$-th factor of the tensor product.
\begin{align*}
\tilde u\tilde A
&=\sum_{\sigma\in S_n}\sum_{i=1}^nT_\sigma(v^{\sigma(1)}\otimes\cdots\otimes v^{\sigma(i)}A\otimes\cdots\otimes v^{\sigma(n)})\\
&=\sum_{\sigma\in S_n}\sum_{i=1}^nT_\sigma(v^{\sigma(1)}\otimes\cdots\otimes Av^{\sigma(i)}\otimes\cdots\otimes v^{\sigma(n)})
 =\tilde A\tilde u
\end{align*}
\end{proof}


\begin{rem}
The inclusion in Proposition~\ref{P.wr} may and may not be strict. Hamming graphs $H(n,m)$ provide examples for both. Taking $m=2$ and $n\ge2$, we have $\Aut^+(K_2\sqprod\cdots\sqprod K_2)=O_n^{-1}\supsetneq S_2\wr S_n=(\Aut^+K_2)\wr S_n$. By \cite{Sch20dt}, we have equality for $m=3$: $\Aut^+(K_3\sqprod\cdots\sqprod K_3)=S_3\wr S_n=(\Aut^+ K_3)\wr S_n$. We are going to prove the equality for any $m>2$ in the case of Hamming graphs.
\end{rem}

\begin{rem}\label{R.wr}
Let $(x_a)_{a=0}^{m-1}$ be a basis of $C(X)$ such that $x_0=1_{C(X)}$ and denote $x_{ai}=1_{C(X)}\otimes\cdots\otimes 1_{C(X)}\otimes x_a\otimes 1_{C(X)}\otimes\cdots\otimes 1_{C(X)}$ as we already did once. Equation~\eqref{eq.hattilu} shows that the action from Proposition~\ref{P.wr} indeed maps $x_{ai}\mapsto\sum_{b,j}x_{bj}\hat u^{bj}_{ai}$.
\end{rem}

\subsection{Proof of Theorem~\ref{T.Hamming}}
Recall that $H(n,m)$ is the Cayley graph of $\Z_m^n$ with respect to the generating set $\{a\epsilon_i\mid i=1,\dots,n,\;a=1,\dots,m-1\}$. We denote by $\tau_\mu$, $\mu\in\Z_m^n$ the irreducible characters of $\Z_m^n$ defined by $\tau_\mu(\alpha)=\gamma^{\alpha_1\mu_1+\cdots+\alpha_n\mu_n}$, where $\gamma$ is some primitive $m$-th root of unity.

As usual, we start with determining the spectrum using Proposition~\ref{P.lambda}:
$$\lambda_\mu=\sum_{i=1}^n\sum_{a=1}^{m-1}\gamma^{\mu_ia}=ml_\mu-n,$$
where $l_\mu=\#\{i\mid \mu_i=0\}$. So, the spectrum contains $n+1$ distinct eigenvalues $n(m-1)$, $n(m-2),\dots,$ $-n$. Denote by $V_0,\dots,V_{n+1}$ the corresponding eigenspaces $V_i=\spanlin\{\tau_\mu\mid l_\mu=n-i\}$. So, for instance $V_0=\spanlin\{\tau_{0,\dots,0}\}$, $V_1=\spanlin\{\tau_{ai}\mid i=1,\dots,n,\;a=1,\dots,m-1\}$, where we denote $\tau_{ai}:=\tau_i^a=\tau_\mu$ for $\mu=(0,\dots,0,a,0,\dots,0)$ -- the $a$ being on the $i$-th place. Those must be invariant subspaces of the fundamental representation of $\Aut^+H(n,m)$.

In Proposition~\ref{P.wr}, we showed that $S_m^+\wr S_n$ acts on $H(n,m)=K_m\sqprod\cdots\sqprod K_m$ via $\tau_{ai}\mapsto\sum_{b,j}\tau_{bj}\otimes\hat u^{bj}_{ai}$ (see Remark~\ref{R.wr}), so $\Aut^+H(n,m)\supset S_m^+\wr S_n$. It remains to show the opposite inclusion.

Denote by $u$ the fundamental representation of $\Aut^+ H(n,m)$. Denote by $\hat u:=\F^{-1}u\F$ the Fourier transform of $u$, that is, the matrix $u$ expressed in the basis of $\tau_\mu$. This matrix decomposes into a direct sum with respect to the invariant subspaces $V_0,V_1,\dots,V_n$ as $\hat u=\hat u^{(0)}\oplus\hat u^{(1)}\oplus\cdots\oplus\hat u^{(n)}$. We denote by $\hat u^{ai}_{bj}$ the entries of $\hat u^{(1)}$. It is enough to show that this matrix $\hat u^{(1)}$ satisfies the relations of the fundamental representation of $\hat S_m^+\wr S_n$. So let us study its intertwiners.

Recall the formula~\eqref{eq.That} for computing the Fourier transform of intertwiners corresponding to block partitions $[\hat T_{b_{k,l}}]^{\nu_1,\dots,\nu_l}_{\mu_1,\dots,\mu_k}=N^{1-k}\delta_{\mu_1+\cdots+\mu_k,\nu_1+\cdots+\nu_l}$. We start by taking $p=\mergepart$ and focus on the entries of $\hat T_{\mergepart}^{(N)}$ corresponding to the invariant subspace $V_1$ and see that $[\hat T_{\mergepart}^{(N)}]^{bj}_{a_1i_1,a_2i_2}=\delta_{i_1i_2j}\delta_{a_1+a_2,b}$. Let us denote $R_{\mergepart}:=\hat T^{(N)}_{\mergeone}\in\Mor(\hat u^{(1)}\otimes\hat u^{(1)},\hat u^{(1)})$ the restriction/projection of $\hat T_{\mergepart}^{(N)}$ onto $V_1$. Let us also denote $R_{\connecterpart}:=R^*_{\mergepart}R_{\mergepart}\in\Mor(\hat u^{(1)\,\otimes 2},\hat u^{(1)\,\otimes 2})$, so that $[R_{\connecterpart}]^{b_1j_1,b_2j_2}_{a_1i_1,a_2i_2}=\delta_{i_1i_2j_1j_2}\delta_{a_1+a_2,b_1+b_2}$.

Next, let us study the intertwiner $\hat T^{(N)}_{\connecterpart}$. Its projection onto $V_1$ can be expressed as
$$N\hat T^{(N)}_{\connecterone}=R_{\connecterpart}+R_{\PAAbb}+R_{\PaBaB}+R_{\PaBBa}.$$
Here, we use the following notation
\begin{align*}
[R_{\PAAbb}]_{a_1i_1,a_2i_2}^{b_1j_1,b_2j_2}&=\delta_{a_1+a_2,0}\delta_{b_1+b_2,0}\delta_{i_1=i_2\neq j_1=j_2},\\
[R_{\PaBaB}]_{a_1i_1,a_2i_2}^{b_1j_1,b_2j_2}&=\delta_{a_1,b_2}\delta_{a_2b_1}\delta_{i_1=j_2\neq i_2=j_1},\\
[R_{\PaBBa}]_{a_1i_1,a_2i_2}^{b_1j_1,b_2j_2}&=\delta_{a_1b_1}\delta_{a_2b_2}\delta_{i_1=j_1\neq i_2=j_2},
\end{align*}
where we use slightly more general notation for the deltas, which is hopefully self-descriptive: For instance, $\delta_{i_1=j_1\neq i_2=j_2}$ equals to one if $i_1=j_1\neq i_2=j_2$ and otherwise it equals to zero. The idea behind the diagrams is that the dashed and dotted blocks indicate the fact that the $i,j$ indices corresponding to the different blocks must not coincide.

We already know that the map $R_{\connecterpart}$ is an intertwiner, which implies that also the sum $R_{\PAAbb}+R_{\PaBaB}+R_{\PaBBa}$ must be an intertwiner. We are going to show that actually each term of this sum is an intertwiner:

First, compute the square of the sum. Obviously, $R_{\PAAbb}R_{\PaBaB}=0=R_{\PAAbb}R_{\PaBBa}$, so it is actually quite easy: $(R_{\PAAbb}+R_{\PaBaB}+R_{\PaBBa})^2=2(m-1)R_{\PAABB}+2R_{\PaBBa}+2R_{\PaBaB}$. Here, $[R_{\PAABB}]_{b_1j_1,b_2j_2}^{a_1i_1,a_2,i_2}=\delta_{a_1+a_2,m}\delta_{b_1+b_2,m}\delta_{a_1a_2b_1b_2}$ (both blocks are dashed so all of the $a_1,a_2,b_1,b_2$ do have to coincide). That was not very helpful actually, but in a similar manner, we can compute the third power: $(R_{\PAAbb}+R_{\PaBaB}+R_{\PaBBa})^3=4(m-1)^2R_{\PAAbb}+4R_{\PaBBa}+4R_{\PaBaB}$. Subtracting four times the original sum, we get $4m(m-2)R_{\PAAbb}$, so $R_{\PAAbb}$ is an intertwiner unless $m=2$. But now $R_{\PaBBa}$ is just a rotation of $R_{\PAAbb}$, so it must also be an intertwiner. Consequently, also $R_{\PaBaB}$ must be an intertwiner and also $R_{\PAABB}$ must be an intertwiner.

Those are all intertwiners we need. Now, we just look at the relations they imply. First, the intertwiner $R_{\PAABB}$ implies the following relation:
\begin{equation}
\sum_{c=1}^{m-1}\hat u^{a_1i_1}_{cj_1}\hat u^{a_2i_2}_{m-c,j_2}\delta_{j_1j_2}\delta_{b_1+b_2,m}=\sum_{d=1}^{m-1}\hat u^{di_1}_{b_1j_1}\hat u^{m-d,i_2}_{b_2j_2}\delta_{i_1i_2}\delta_{a_1+a_2,m}
\end{equation}
So, we can define $w^i_j:=\sum_c\hat u^{ai}_{cj}\hat u^{ai}_{m-c,j}=\sum_d \hat u^{di}_{bj}u^{m-d,i}_{bj}$ (thanks to the relation above, all the sums are equal regardless of the choice of $a,b$). Let us also define $\hat v^{ai}_b:=\sum_ju^{ai}_{bj}$ (compare with Remark~\ref{R.fwr}(f)).

Now it remains to derive the following relations:

{
\def\eq#1=#2\\{\hbox{%
\refstepcounter{equation}%
\rlap{(\theequation)}%
\kern7em\llap{$#1$}%
${}=#2$
}}
\def\expl#1{\openup-\jot
\hskip \displaywidth minus 1fill
\vcenter{\advance\hsize by -15em\par\noindent
#1
}}
\def\eqset#1{%
\left.\kern-\nulldelimiterspace\vcenter{\hsize=14em #1}\quad\right\}
}

$$\displaylines{
\eqset{
\eq \sum_jw^i_j=1                               \label{eq.Hw1}\\ 
\eq w^i_jw^k_l\delta_{jl}=w^i_jw^k_l\delta_{ik} \label{eq.Hw2}\\
\eq w^i_jw^k_l=w^k_lw^i_j                       \label{eq.Hw3}\\
}
\expl{$w$ satisfies the relations of $S_n$ (we use the intertwiner relations of $T_{\singleton}^{(m)}$, $T_{\connecterpart}^{(m)}$, and $T_{\crosspart}^{(m)}$),}\cr
\eq v^{ak}_bw^i_j=w^i_jv^{ak}_b                 \label{eq.Hvw}\\
\expl{$w^i_j$ commute with everything,}\cr
\eq v^{a_1i}_bv^{a_2i}_b=v^{a_1i}_b\delta_{a_1+a_2,n}\label{eq.Hv1}\\
\expl{$v^i$ satisfy the relations of $\hat S_m$ (we use the intertwiner relation of $T_{\mergepart}^{(n)}$),}\cr
\eq v^{ai}_bv^{cj}_d=v^{cj}_dv^{ai}_b           \label{eq.Hv2}\\
\expl{entries of $v^i$ commute with entries of $v^j$ for $i\neq j$,}\cr
\eq u^{ai}_{bj}=v^{ai}_bw^i_j                   \label{eq.Huvw}\\
\expl{$u$ indeed has the correct structure.}
}
$$
}

The twisted orthogonality relation corresponding to the intertwiner $\hat T^{(N)}_{\pairpart}$ looks as follows:
$$\sum_{b,j}u^{a_1i_1}_{bj}u^{a_2i_2}_{m-b,j}=\delta_{a_1+a_2,m}\delta_{i_1i_2}$$
This, in particular, implies Relation~\eqref{eq.Hw1}.

We write down the relation corresponding to $R_{\PaBaB}$:
$$u^{a_1i_1}_{b_2j_2}u^{a_2i_2}_{b_1j_1}\delta_{j_1\neq j_2}=u^{a_2i_2}_{b_1j_1}u^{a_1i_1}_{b_2j_2}\delta_{i_1\neq i_2}$$
This implies two things. First,
\begin{equation}\label{eq.Hu1}
u^{ai}_{bj}u^{ck}_{dl}=u^{ck}_{dl}u^{ai}_{bj}\qquad\text{whenever $i\neq\nobreak k$ or $j\neq l$.}
\end{equation}
Secondly, (and for this we might as well use the relation corresponding to $R_{\PaBBa}$),
\begin{equation}\label{eq.Hu2}
u^{ai}_{bj}u^{ck}_{dj}=0\quad\text{for $i\neq k$,}\qquad u^{ai}_{bj}u^{ci}_{dl}=0\quad\text{for $j\neq l$.}
\end{equation}

The latter remark allows us to check Relation~\eqref{eq.Hw2}. Assume $j\neq l$, then
$$w^i_jw^i_l=\sum_{c,d}\hat u^{ai}_{cj}\hat u^{ai}_{m-c,j}\hat u^{ai}_{dl}\hat u^{ai}_{m-d,l}=0.$$
Similarly, we derive that $w^i_jw^k_j=0$ for $i\neq k$.

Relation~\eqref{eq.Hu1} then implies Relation~\eqref{eq.Hw3}, i.e.\ the commutativity $w^i_jw^k_l=w^k_lw^i_j$. Indeed, notice that $w^i_j$ obviously commutes with itself, so we can assume that $i\neq k$ or $j\neq l$ and then it is a direct application of \eqref{eq.Hu1}.

We can also use Relation~\eqref{eq.Hu2} to check \eqref{eq.Huvw}:
$$v^{ai}_bw^i_j=\sum_ku^{ai}_{bk}\sum_cu^{ai}_{cj}u^{ai}_{m-c,j}=u^{ai}_{bj}\sum_{c,l}u^{ai}_{cl}u^{ai}_{m-c,l}=u^{ai}_{bj}.$$

Similarly, we can derive $u^{ai}_{bj}=w^i_j v^{ai}_b$, which proves part of Relation~\eqref{eq.Hvw}. To finish the proof of this relation, assume that $j\neq k$ and compute
$$v^{ak}_bw^i_j=\sum_lu^{ak}_{bl}\sum_cu^{ai}_{cj}u^{ai}_{m-c,j}=\sum_lu^{ak}_{bl}\sum_cu^{ai}_{cj}u^{ai}_{m-c,j}=w^i_jv^{ak}_b.$$

Relation~\eqref{eq.Hv2} goes the same way.

Finally, Relation~\eqref{eq.Hv1} follows directly from the relation corresponding to $R_{\mergepart}$, which reads
$$u^{a_1i_1}_{bj}u^{a_2i_2}_{bj}=u^{a_1+a_2,i_1}_{bj}\delta_{b_1b_2}.\eqno\qed$$

\bibliographystyle{halpha}
\bibliography{mybase}

\end{document}